\documentclass{amsart}
\usepackage{amssymb}
\usepackage{booktabs}
\usepackage{mathtools}
\usepackage{minibox}
\usepackage{multirow}
\usepackage{rotating}
\usepackage{psfrag}
\usepackage{float}
\usepackage{xr-hyper}
\usepackage[colorlinks]{hyperref}


\thanks{D.S. was partially supported by CNPq 307617/2016-5, CNPq 430351/2018-6 and FAPESP Projeto Tem\'atico 2017/06463-3.}

\title[Tranfers operators, Atomic Decomposition and the Bestiary]{Transfer operators,   atomic decomposition and the Bestiary}

\author[D. Smania]{Daniel Smania}
\address{Departamento de Matem\'atica, Instituto de Ci\^encias Matem\'aticas e da Computa\c{c}\~ao-Universidade de S\~ao Paulo (ICMC/USP), Caixa Postal 668, S\~ao Carlos-SP, Brazil.}
\email{smania@icmc.usp.br}
\urladdr{http://conteudo.icmc.usp.br/pessoas/smania/}

\newtheorem{theorem}{Theorem}[section]
\newtheorem{corollary}{Corollary}[section]

\newtheorem{lemma}[theorem]{Lemma}
\newtheorem{proposition}[theorem]{Proposition}

\theoremstyle{definition}

\newtheorem{remark}[theorem]{Remark}

\usepackage{etoolbox}
\usepackage{constants} 

\newconstantfamily{name}{
symbol=C,
format=\secdot,
}

\newconstantfamily{namec}{
symbol=\lambda,
format=\secdot,
}

\newconstantfamily{A}{
symbol=A,
format=\arabic,
}

\newcommand{\secdot}[1]{\arabic{#1}}
\newconstantfamily{c}{
symbol=\lambda,
format=\secdot,
reset={},
}

\renewconstantfamily{normal}{
symbol=C,
format=\secdot,
}

\newconstantfamily{e}{
symbol=\theta,
format=\arabic,
}

\makeatletter
\newcommand{\Cll}[2][normal]{%
\global\@namedef{cstt@#1@#2}{}%
\@ifundefined{cstt@name@#2} {\@ifundefined{cstt@namec@#2}{ {\Cl[#1]{#2}}}{ {\lambda_{_{#2}}}}}{{{C_{_{#2}}}}} %
}
\newcommand{\Crr}[1]{\@ifundefined{cstt@name@#1}{   \@ifundefined{cstt@namec@#1}{\Cr{#1}}{\lambda_{_{#1}}}}{C_{_{#1}}}}
\makeatother

\usepackage[colorinlistoftodos, textwidth=4cm, shadow]{todonotes}

\hypersetup{ colorlinks=true, pdftitle={Transfer operators and  atomic decomposition}, pdfsubject={ transfer operator, atomic decomposition, Besov space,  Ruelle,  Perron-Frobenious,  quasi-compact, Lasota-Yorke, decay of correlations, expanding map}, pdfauthor={Alexander Arbieto and Daniel Smania},pdfkeywords={ transfer operator, atomic decomposition, Besov space,  Ruelle,  Perron-Frobenious,  quasi-compact, Lasota-Yorke, decay of correlations, expanding map}}

\begin{document}

\begin{abstract}   Arbieto and S. recently used atomic decomposition to study transfer operators. We give a long list of old and new expanding dynamical systems for which those results  can be applied, obtaining  the quasi-compactness of  transfer operator  acting  on  Besov spaces of  measure spaces with a good grid.\end{abstract}

\maketitle

\setcounter{tocdepth}{2}
\tableofcontents

\vspace{1cm}
\centerline{ \bf I. INTRODUCTION.}
\addcontentsline{toc}{chapter}{\bf I. INTRODUCTION.}
\vspace{1cm}

In S.  \cite{smania-besov} we defined Besov spaces on measure spaces endowed with a "good grid". This allowed Arbieto and S. \cite{smania-transfer} to give sufficient conditions for  the transfer operator  of maps acting on these measure spaces to be quasi-compact and to satisfy the Lasota-Yorke inequality.  Many nice statistical properties of Besov observables follow.

Bestiaries were popular in the Middle Ages in Europe. Those were a compendium   of  wonderful animals. We offer a compendium of exquisite piecewise expanding maps, and we prove (sometimes conditioned to a priori estimate) the quasi-compactness and Lasota-Yorke inequalities for their transfer operator acting on Besov spaces. We  list  most of the examples in Table 1.  

We order our presentation in such way we go from the simplest one, linear expanding maps  on  the circle,  to the most complex example, piecewise expanding maps on $\mathbb{R}^D$. 

Our first examples are Markovian expanding maps and conformal expanding maps. They do not have discontinuities and their branches have large images. This allows us to give  precise estimates to the essential spectral radius for the transfer operator. These class of examples includes subshifts of finite type and  hyperbolic rational maps acting on its Julia sets.

Intervals maps are our next class of examples. Those include piecewise $C^{1+}$-diffeomorphism  expanding maps, piecewise Bi-Lipchitz maps with $p$-bounded variation jacobian and  Lorenz maps. We also obtain results for piecewise Bi-Lipchitz maps with certain {\it Besov} jacobians, a very general class of potentials, but we need a priori estimate in this case.

The last example is given by  generic piecewise $C^{1+}$-diffeomorphisms on $\mathbb{R}^D$.  This is a more complex situation because large  iterations  deforms shapes  in a more extreme way.

We focus ourselves to obtain the quasi-compactness and the Lasota-Yorke inequality for the transfer operator acting on  Besov spaces. With the exception  of the a recent result by Nakano and  Sakamoto \cite{ns} (see also Baladi and  Holschneider \cite{bh}) for smooth expanding maps on manifolds (no discontinuities) and the the work on Thomine \cite{thomine} on  transfer operators of piecewise $C^{1+}$ expanding maps on manifolds acting  on Sobolev spaces, all  our  results are new,  in particular the  results for Besov spaces on phases spaces with very low regularity, as either symbolic spaces or hyperbolic Julia sets. 

Moreover often these results imply that Besov observables have nice statistical properties, as the almost sure invariance principle and  exponential decay of correlations. We refer to Arbieto and S.  \cite{smania-transfer} for these consequences. 

Another consequence for  most of the  examples  here is that the support of every absolutely continuous invariant  measure is an open subset of phase space (up to a set of zero measure). This, as  far as we know, it is also a new result in some cases, as for  generic piecewise  $C^{1+}$ expanding maps in $\mathbb{R}^D$.

One may wonder if  the atomic decomposition methods in \cite{smania-transfer} could be applied to more classes of maps, as measure expanding solenoidal attractors studied by  Tsujii \cite{sole1},  Avila,  Gou\"ezel,  Tsujii \cite{sole2}, Bam\'on, Kiwi, Rivera-Letelier, Urz\'ua \cite{sole3}, and  maps with critical points. In the latter class  there is a previous study of the Besov  regularity of the density of  invariant measures by Chazottes,  Collet and Schmitt \cite{cha}.

{\small
\begin{table*}[h]
    \centering
        \begin{tabular}{c c c c c}
            \toprule
            \midrule
                 {\tiny Class} &  {\tiny Example}  &  {\tiny p }   & \multirow{1}{*}{\parbox{1.8cm}{\tiny It needs a  priori \\ estimate?} }& \multirow{1}{*}{\parbox{1.8cm}{\tiny Map has  \\ discontinuites?}}  \\ \cmidrule{1-5}
                             \multicolumn{1}{l}{\tiny H\"older jacobian}  & 
                \multicolumn{1}{l}{\tiny Markovian  Maps}& {\tiny $[1,\infty)$}  & {\tiny No}  & {\tiny "No" } \\
                  \cmidrule{1-5} 
                                      \multicolumn{1}{l}{\tiny Complex analytic  map}  &
                \multicolumn{1}{l}{ \multirow{1}{*}{\parbox{2.5cm}{\tiny Conformal \\ expanding repellers} }}& {\tiny $[1,\infty)$}  & {\tiny No}& {\tiny No}   \\
                    \cmidrule{1-5}
                 \multicolumn{1}{l}{ \multirow{4}{*}{\tiny Interval  maps}}& \multicolumn{1}{l}{\multirow{1}{*}{\parbox{2.5cm}{\tiny Bounded variation \\ jacobian} }}& {\tiny $\sim 1$} & {\tiny No} & {\tiny Yes}  \\     
    \cmidrule{2-5}                 
                    \multicolumn{1}{c}{}  &  \multicolumn{1}{l}{\multirow{1}{*}{\parbox{2.5cm}{\tiny  Piecewise \\ $C^{1+}$-smooth maps}}}  & {\tiny $\sim 1$} & {\tiny No} & {\tiny Yes}  \\   
                    \cmidrule{2-5}         \multicolumn{1}{c}{}  &   \multicolumn{1}{l}{\tiny jacobian in $\mathcal{B}^{1/p}_{p,\infty}$} & {\tiny $[1,\infty)$}  & {\tiny Yes} & {\tiny Yes}  \\ 
                    \cmidrule{2-5}    
                  \multicolumn{1}{c}{}  &   \multicolumn{1}{l}{\tiny Lorenz maps} & {\tiny $1$}  & {\tiny No} & {\tiny Yes}  \\ 
                    \cmidrule{2-5}    
                  \multicolumn{1}{c}{}  &   \multicolumn{1}{l}{\tiny Tent family } & {\tiny $[1,\infty)$}  & {\tiny No} & {\tiny No}  \\ 
                  \midrule
                    \multicolumn{1}{l}{\tiny Cowieson-type  maps}    &
                \multicolumn{1}{l}{\tiny $C^{1+}$ Piecewise Smooth Maps}& {\tiny $\sim 1$}  & {\tiny Generic}  & {\tiny Yes}  \\
            \midrule
            \bottomrule
        \end{tabular}
        \centering
        \vspace{3mm}
        \caption{The Bestiary. The first column  describes for each values of $p$ the transfer operator is quasi-compact on $\mathcal{B}^s_{p,q}$.  The second column tells us if we need an a priori estimate as an assumption. The last column  says if we allow discontinuities on the maps under consideration. }
        \label{tab:sam_count}
\end{table*}
}

\newpage
 \vspace{1cm}
\centerline{ \bf II.  THE MAIN INGREDIENTS}
\addcontentsline{toc}{chapter}{\bf II.  THE MAIN INGREDIENTS}
\vspace{1cm}

\section{Regular Branches}

\subsection{Bilipschitz maps on Ahlfors-regular quasi-metric spaces.} Let $I$ be a metric space with a quasi-metric $d$ and a finite  measure $m$. Then $(I,d,M)$ is an  Ahlfors-regular quasi-metric space if there  is $D$, $\Cll{iop}$  and $r_0$ with the following property.
  For every $x\in I$ and $r \in (0,r_0)$ 
\begin{equation}\label{geo}   \frac{1}{\Crr{iop}} r^D\leq    m(B_d(x,r))\leq \Crr{iop} r^D.\end{equation}
There is a good grid $\mathcal{P}$ in $I$ and  there is $\eta, \Crr{de111d}, \Crr{de1111d},\Crr{de00d},\Crr{ood}\geq 0$ and $\Crr{de000d}\in (0,1)$ with the following property (see Proposition \ref{homo-fgh} in  \cite{smania-homo}). For every $Q\in \mathcal{P}^k$ there is $z_Q\in Q$ satisfying 
\begin{equation}\label{ppd} B_{d}(z_Q,\Cll{de111d} \Crr{de000d}^k)  \subset Q,\end{equation}
\begin{equation}\label{pp2d} diam_d \  Q \leq \Cll{de1111d} \Crr{de000d}^k  \end{equation}
and
\begin{equation}\label{pppd}  m\{ x\in Q\colon \  d(x,I\setminus Q)\leq \Cll{de00d} t  \Cll[c]{de000d}^k    \}\leq \Cll{ood} t^{\eta} m(Q).\end{equation} 
We can consider the space $\mathcal{B}^s_{p,q}$ associated with $(I,m,\mathcal{P})$. If $\eta > sp$ then $\mathcal{B}^s_{p,q}$ indeed does not depend on the particular good grid we choose and it is called the  Besov space of $(I,m)$.  A classical example to keep in mind is $[0,1]^n$ endowed with the euclidean metric and the Lebesgue measure. In this case we can take the usual dyadic good grid and $\eta=1$.

\begin{proposition} \label{lip} Suppose  $Dsp < \hat{\eta}$. There is $\Cll{abs3ww}$ with the following property.  Let $\Omega, \Omega' \subset I$ be open sets and  $h\colon \Omega \rightarrow  \Omega'$ be a bilipschitz map.   In particular there  is $\Cll{lm11},\Cll{lm1} > 0$ such that for every $x,y \in  \Omega$
\begin{equation} \label{lip2} \Crr{lm11}d(x,y) \leq   d(h(x),h(y)) \leq \Crr{lm1}d(x,y)\end{equation}
and  there  is $\Cll{lms}, \Cll{lm} > 0$ such that for every measurable set $A \subset \Omega$
\begin{equation} \label{mc}  \Crr{lms} \leq    \frac{|h(A)|}{|A|} \leq \Crr{lm}.\end{equation}Let $Q\subset I$ be an open subset such that there is $z_Q\in Q$ satisfying 
\begin{equation}\label{pphe} B_{d}(z_Q,\Cll{de1w11a} diam_d \  Q )  \subset Q,\end{equation}
and 
\begin{equation}\label{pphe2} m\{ x\in Q\colon \  d(x,I\setminus Q)\leq \Cll{de0w0x} t  \ diam_d \  Q   \}\leq \Cll{owox} t^{\hat{\eta}} m(Q)\end{equation}
for some $\Crr{de1w11a}$, $\Crr{de0w0x}$, $\Crr{owox}$, and  $\hat{\eta} > 0$.  Then there is $z_{h(Q)}$ such that 
$$B_d(z_{h(Q)}, \Crr{de1w11a} \frac{\Crr{lm11}}{\Crr{lm1}}  \ diam_d \ h(Q) )\subset h(Q).$$
and
$$ m(\{x\in h(Q)\colon \  d(x,I\setminus h(Q))  \leq \Crr{de0w0x}\frac{\Crr{lm11}}{\Crr{lm1}} t \  diam_d \ h(Q)     \})\leq  \Crr{owox} \frac{ \Crr{lm}}{\Crr{lms}}   t^{\hat{\eta}} m(h(Q)).$$
Moreover  $h(Q)$ is a 
 $(1-sp, \Cll{abs3}, \Crr{de000d}^{\hat{\eta}-Dsp})$-regular domain,  with 
 $$\Crr{abs3}= \Crr{abs3ww}  \frac{ \Crr{owox} \frac{ \Crr{lm}}{\Crr{lms}}}{\big( \Crr{de0w0x}\frac{\Crr{lm11}}{\Crr{lm1}}\big)^{\hat{\eta}} \big( \Crr{de1w11a} \frac{\Crr{lm11}}{\Crr{lm1}} \big)^D}.$$ 
\end{proposition} 
\begin{proof}   Note    that  $z_{h(Q)}=h(x_Q)$ satisfies
\begin{equation}\label{hip1} B_d(z_{h(Q)},\Crr{de1w11a} \frac{\Crr{lm11}}{\Crr{lm1}}  \ diam_d \ h(Q) )\subset h(Q).\end{equation}
Suppose that $x \in h(Q)$ and it  satisfies
$$d(x,I\setminus h(Q))\leq  \Crr{de0w0x}\frac{\Crr{lm11}}{\Crr{lm1}} t  \ diam_d \ h(Q).$$
Then
$$d(h^{-1}(x),I\setminus Q)\leq  \Crr{de0w0x} \  t \ diam_d \ Q.$$
Consequently
$$h^{-1}\{ x\in h(Q)\colon \  d(x,I\setminus h(Q))\leq \Crr{de0w0x}\frac{\Crr{lm11}}{\Crr{lm1}}  \ t  \ diam_d \ h(Q)   \}$$
is contained in $$\{ x\in Q\colon \  d(x,I\setminus Q)\leq \Crr{de0w0x} t \  diam_d \ Q   \},$$
so
\begin{align*} 
&m(h^{-1}\{ x\in h(Q)\colon \  d(x,I\setminus h(Q))\leq \Crr{de0w0x}\frac{\Crr{lm11}}{\Crr{lm1}} t  \ diam_d \ h(Q)    \})\\
&\leq  m(\{ x\in Q\colon \  d(x,I\setminus Q)\leq \Crr{de0w0x}t  \ diam_d \ Q    \})\\
&\leq \Crr{owox} t^{\hat{\eta}} m(Q)\\
&\leq  \frac{\Crr{owox}}{\Crr{lms}}  t^{\hat{\eta}} m(h(Q)).
\end{align*} 
and finally
\begin{align}\label{hip2} &m(\{x\in h(Q)\colon \  d(x,I\setminus h(Q))  \leq \Crr{de0w0x}\frac{\Crr{lm11}}{\Crr{lm1}}   t \  diam_d \ h(Q)     \}) \nonumber \\
&\leq \Crr{lm} m(h^{-1}\{ x\in h(Q)\colon \  d(x,I\setminus h(Q))  \leq \Crr{de0w0x}\frac{\Crr{lm11}}{\Crr{lm1}}   t \  diam_d \ h(Q)     \})\nonumber \\
&\leq  \Crr{owox} \frac{ \Crr{lm}}{\Crr{lms}} t^{\hat{\eta}} m(h(Q)).
\end{align}
By (\ref{hip1}), (\ref{hip2}) and Proposition \ref{homo-indicator} in  \cite{smania-homo} it follows that $h(Q)$ is a $(1-sp, \Crr{abs3}, \Crr{de000d}^{\hat{\eta}-Dsp})$-regular domain.
\end{proof}

\begin{corollary} There is $\Cll{ggg}$ such that  for every $Q\in \mathcal{P}$   the set $h(Q)$ is 
a  $(1-sp, \Cll{abs33}, \Crr{de000d}^{\eta-Dsp})$-regular domain, with
$$\Crr{abs33}= \Crr{abs3ww}  \frac{\Crr{ood} \frac{ \Crr{lm}}{\Crr{lms}}}{\big( \frac{\Crr{de00d}}{\Crr{de1111d}}\frac{\Crr{lm11}}{\Crr{lm1}}\big)^{\eta} \big( \frac{\Crr{de111d}}{\Crr{de1111d}}  \frac{\Crr{lm11}}{\Crr{lm1}} \big)^D}.$$ 
\end{corollary} 
\begin{proof} Note that 
$$B_{d}(z_Q, \frac{\Crr{de111d}}{\Crr{de1111d}}  \ diam_d Q )  \subset Q,$$
and
$$m\{ x\in Q\colon \  d(x,I\setminus Q)\leq \frac{\Crr{de00d}}{\Crr{de1111d}} t  \ diam_d Q    \}\leq \Crr{ood} t^{\eta} m(Q).$$

\subsection{$C^{1}$-diffeomorphisms on $\mathbb{R}^D$.} A $D$-cube in $\mathbb{R}^D$ is a set $K\subset \mathbb{R}^D$ defined as
$$K =\{ x_0 + \sum_{i=1}^D \alpha_i v_i, \ \alpha _i \in [0,1]\},$$
where $x_0 \in \mathbb{R}^D$ and $B=\{v_1,\dots, v_n\}$ is a basis of $\mathbb{R}^D$. We can consider the Lebesgue measure $m_K$ on $K$ normalized such that $m_K(K)=1$ and the {\it dyadic} grid $\mathcal{D}_K$ defined as $Q\in \mathcal{D}_K^m$ if there are integers  $0\leq j_i < 2^m$ such that 
$$Q=\{  x_0 + \sum_{i=1}^D \alpha_i v_i, \ \alpha _i \in [\frac{j_i}{2^m}, \frac{j_i+1}{2^m}]\}.$$
Of course $\mathcal{D}_K$ is good grid in $(K,m_K)$.  If we consider the metric $d_K$ on $K$ such that  $d_K(x,y)=|x-y|_K$, where $|\cdot|_k$ comes from an inner product that turns $B$ into an orthonormal basis then $(K, d_k, m_k)$ is an $D$-dimensional Alhfors regular metric space and $\mathcal{P}=\mathcal{D}_k$ is good grid on it that satisfies  (\ref{ppd}), (\ref{pp2d}) and (\ref{pppd}), with $\eta=1$ and  the constants that appears there may be chosen such that  they {\it do not depend on the the chosen orthonormal basis and $D$-cube $K$}. 

The  corresponding Besov space $\mathcal{B}^s_{p,q}(K,m,\mathcal{D}_K)$, with $0 < s < 1/p$, $p\in [1,\infty)$ and $q\in [1,\infty]$,  coincides with the classical  Besov space in the homogeneous space $(K,m_K)$ (see \cite{smania-homo}). 

\begin{proposition} \label{ant}  Let  $(K,m_K,\mathcal{P})$ be  a measure space with a good grid where $K$ is a compact subset of $\mathbb{R}^D$, $m_K$ is the Lebesgue measure up to a scaling factor and  $\mathcal{P}$ is a good grid  satisfying (\ref{geo}),(\ref{ppd}), (\ref{pp2d}) and (\ref{pppd}) taking $d$ as the  euclidean distance multiplied by a factor. Then there is $\Cll{esss}$, that depends only on the constants that appears in these conditions, such that the following holds. Let 
$$0 < \beta_1 \leq \beta_2 \leq \cdots \leq \beta_D$$
and 
$$Q=    F(\Pi_{i=1}^D [0,\beta_i]),$$
where $F$ is an isometry of $\mathbb{R}^D$.  Suppose that $F(Q)\subset K$. Then   $Q$ is a $(1-sp,\Crr{esss}(\Pi_{i\neq 1} \beta_i/\beta_1)^{sp}, \Crr{de000d}^{1-Dsp})$-regular domain on $(K,m_K,\mathcal{P})$. \end{proposition} 
\begin{proof} Note that if $\Cll{au}^{-1}d$ is the euclidean distance then (\ref{geo}) implies 
$$ \frac{\Crr{auu}}{\Crr{au}^D\Crr{iop}}    \leq \frac{m(A)}{m_K(A)}\leq \frac{\Crr{auu}}{\Crr{au}^D} \Crr{iop}$$
for every measurable set $A$, where the constant $\Cll{auu}$ is universal and $m$ is the Lebesgue measure (without any normalization). Moreover 
$$\frac{\Crr{de111d}}{ \Crr{au}} \Crr{de000d}^{k_0(Q)} \leq   \beta_1\leq \frac{2\Crr{de1111d}  }{\Crr{de000d}\Crr{au}} \Crr{de000d}^{k_0(Q)} .$$
We claim that 
$$m(x\in \mathbb{R}^D\colon \ d(x,\mathbb{R}^D\setminus Q) < t \Crr{au}  \beta_1)\leq 2Dtm(Q).$$
It is enough to prove the claim for the case $F=Id$.    We have that 
$$\{ x\in \mathbb{R}^D\colon \ d(x,\mathbb{R}^D\setminus Q) \leq  t \Crr{au} \beta_1 \}$$
is contained in
$$\cup_j  \{(x_1,\dots,x_D)\colon \ x_i \in [0,\beta_i] \text{ for $i\neq j$, and } x_j \in [0,t\beta_1]\cup [\beta_j- t\beta_1,\beta_j]   \},$$
so
\begin{align*} &m(x\in \mathbb{R}^D\colon \ d(x,\mathbb{R}^D\setminus Q) \leq  t \Crr{au} \beta_1) \\
& \leq \sum_j 2t\beta_1 \Pi_{i\neq j} \beta_i\\
& \leq 2t \big( \sum_j \frac{\beta_1}{\beta_j}  \big) \Pi_{i} \beta_i\\
& \leq 2Dtm(Q).
\end{align*} 
This proves the claim. Define $\mathcal{F}^k(Q)\subset \mathcal{P}^k$ recursively as
$$\mathcal{F}^{k_0}(Q)=\{  P \in \mathcal{P}^{k_0}\colon \ P\subset Q\}$$
and 
$$\mathcal{F}^{k+1}(Q)=\{  P \in \mathcal{P}^{k}\colon \ P\subset Q\setminus \cup_{j\leq k} \cup_{W\in \mathcal{F}^j(Q)} W \}.$$
Note that if $d(x,\mathbb{R}^n\setminus Q) >  \Crr{de1111d} \Crr{de000d}^{k-1}$ then  there is 
$W\in  \cup_{j\leq k-1} \mathcal{F}^j(Q)$ such that $x\in W$, so
\begin{align*} \sum_{P \in \mathcal{F}^k(Q)} m(P)&\leq m(x\in \mathbb{R}^D\colon \ d(x,\mathbb{R}^D\setminus Q) <  \Crr{de1111d} \Crr{de000d}^{k-1})\\
&\leq 2D \Crr{de1111d}  \frac{\Crr{de000d}^{k-1}}{\Crr{au} \beta_1} m(Q) \leq \frac{2D \Crr{de1111d}}{ \Crr{de111d} \Crr{de000d}}\Crr{de000d}^{k-k_0(Q)}m(Q).
\end{align*} 
Since
$$\frac{\Crr{auu}}{\Crr{au}^D\Crr{iop}^2}  \Crr{de111d}^D \Crr{de000d}^{Dk}\leq m(P)\leq \frac{\Crr{auu}\Crr{iop}^2}{\Crr{au}^D}  \Crr{de1111d}^D \Crr{de000d}^{Dk}$$
for every $P\in \mathcal{F}^k(Q)$,
we have that 
$$\# \mathcal{F}^k(Q)\leq  \frac{2D \Crr{de1111d} \Crr{au}^D}{\Crr{auu} \Crr{de111d}^{1+D}\Crr{de000d}}\Crr{de000d}^{k-k_0(Q)} \Crr{iop}^2 \Crr{de000d}^{-Dk}m(Q).$$
Denote
$$\Cll{ess}= \frac{2D \Crr{de1111d}\Crr{iop}^2}{ \Crr{de111d}^{1+D}\Crr{de000d}}.$$ We conclude that
\begin{align*} &\sum_{P \in \mathcal{F}^k(Q)} m_K(P)^{1-sp} \\
& \leq \big(  \frac{\Crr{au}^D\Crr{iop}}{\Crr{auu}}\big)^{1-sp}   \sum_{P \in \mathcal{F}^k(Q)} m(P)^{1-sp}  \\
&\leq \big(  \frac{\Crr{au}^D\Crr{iop}}{\Crr{auu}}\big)^{1-sp}   \Crr{ess} \frac{\Crr{au}^D}{\Crr{auu}}\Crr{de000d}^{k-k_0(Q)}   \Crr{de000d}^{-Dk} \Crr{de1111d}^{D(1-sp)}\big( \frac{\Crr{auu}}{\Crr{au}^D} \Crr{iop}^2 \big)^{1-sp}\Crr{de000d}^{Dk(1-sp)}  m(Q)\\
&\leq \Crr{ess}\Crr{de1111d}^{D(1-sp)}  \Crr{iop}^{3(1-sp)} \frac{\Crr{au}^D}{\Crr{auu}} \Crr{de000d}^{k-k_0(Q)} \Crr{de000d}^{-Dsp(k-k_0(Q))}\Crr{de000d}^{-Dspk_0(Q)} m(Q)^{sp}   m(Q)^{1-sp}\\
&\leq \Crr{ess}\Crr{de1111d}^{D(1-sp)}  \Crr{iop}^{3(1-sp)} \frac{\Crr{au}^D}{\Crr{auu}}\Crr{de000d}^{(1-Dsp)(k-k_0(Q))} \frac{2^{Dsp}\Crr{de1111d}^{Dsp}  }{\Crr{de000d}^{Dsp}\Crr{au}^{Dsp}}  \Big( \Pi_{i\neq 1}  \frac{\beta_i}{\beta_1}  \Big)^{sp} \big(\frac{\Crr{auu}}{\Crr{au}^D} \Crr{iop} \big)^{1-sp}  m_K(Q)^{1-sp}\\
&\leq \Crr{esss} \Crr{de000d}^{(1-Dsp)(k-k_0(Q))}  \Big( \Pi_{i\neq 1}  \frac{\beta_i}{\beta_1}  \Big)^{sp}  m_K(Q)^{1-sp},
\end{align*} 
where it is worth noting  that  $\Crr{esss}$ does not depend on the normalising factor $\Crr{au}$. 
\end{proof}

\begin{proposition} \label{qwe}Let $(K,m_K,\mathcal{P})$ be  a measure space with a good grid as in Proposition \ref{ant}.  For every  small $\delta$ there is $\Cll{essss}$ such that the following holds.  Let $A\colon \mathbb{R}^D\rightarrow \mathbb{R}^D$ be an invertible  linear transformation. Let 
$$0 < \alpha_1 \leq  \dots \leq  \alpha_D $$
be such that $\{\alpha_i^2\}_i$ are the  eigenvalues of $AA^\star$, repeating the eigenvalues  the  number of times  corresponding to  its multiplicities. Let $W$ be a bounded open set satisfying (\ref{pphe}) and (\ref{pphe2}), with $Dsp < \min\{ \hat{\eta},\eta\}$.  If $A(W)\subset K$ then $A(W)$ is a $$(1-sp, \Crr{essss}  (\Pi_{i\neq 1}  \frac{\alpha_i}{\alpha_1})^{sp}, \Crr{de000d}^{(1-\delta)(\hat{\eta}-Dsp)})$$ regular domain in $(K,m,\mathcal{P})$. The constant $\Crr{essss}$ depends only on  $\delta$ and the constants in (\ref{ppd}), (\ref{pp2d}) and (\ref{pppd}) for the good grid $\mathcal{P}$ considering $d$ as either the euclidean distance or a multiply of it. 
\end{proposition}
\begin{proof}Let $\hat{B}=\{v_1,\dots, v_n\}$ be a orthonormal basis of $\mathbb{R}^D$ such that $AA^\star v_i = \alpha_i^2  v_i$.  Then $B=\{ A(v_1)/\alpha_1 ,\dots, A(v_D)/\alpha_D\}$ is also an orthonormal  basis. Consider a $D$-cube $\hat{K}$ with sides parallels to the basis $\hat{B}$ such that $Q \subset \hat{K}$. Then $(\hat{K},d_{\hat{K}},\mathcal{D}_{\hat{K}})$ is measure space with a good grid satisfying (\ref{ppd}), (\ref{pp2d}) and (\ref{pppd}) with $\eta=1$, $\Crr{de000d}=1/2$ and the other constants that appears there may be chosen such that  they do  not depend on the chosen orthonormal basis and $D$-cube $\hat{K}$. By Proposition \ref{lip} (take $h=Id$) we have that $W$ is a $(1-sp, \Crr{abs3}, (1/2)^{\hat{\eta}-Dsp})$-regular domain in $(\hat{K},d_{\hat{K}},\mathcal{D}_{\hat{K}})$,  with 
 $$\Crr{abs3}= \Crr{abs3ww}  \frac{ \Crr{owox}}{\Crr{de0w0x}^{\hat{\eta}} \Crr{de1w11a}^D}.$$ 
 and $\Crr{abs3ww}$ does not depend on the chosen orthonormal basis and $D$-cube $\hat{K}$, that is , we can find families  $\hat{\mathcal{F}}^j(W)\subset \mathcal{D}^j_{\hat{K}}$ such that 
$$ \sum_j \sum_{P\in\hat{\mathcal{F}}^j(W)} 1_P= 1_W$$ and
\begin{equation}\label{v45} \sum_{P\in \hat{\mathcal{F}}^j(W)} m_{\hat{K}}(P)^{1-sp} \leq \Crr{abs3} (1/2)^{(j-k_0(W,\mathcal{D}_{\hat{K}} ))(\hat{\eta}-Dsp)}m_{\hat{K}}(W)^{1-sp}.\end{equation}
Note that for every $P\in \cup_j \hat{\mathcal{F}}^j(W)$ we have that $A(P)$ is a set as in Proposition \ref{ant}, where $\beta_i = \alpha_i c$, for some $c> 0$. So $A(P)$ is $(1-sp,\Crr{esss}\Pi_{i\neq 1} \alpha_i/\alpha_1, \Crr{de000d}^{1-Dsp})$-regular domain on $(K,m_K,\mathcal{P})$, so there are families $\mathcal{F}^j(A(P))\subset \mathcal{P}^j$ such that 
$$\sum_j \sum_{R\in \mathcal{P}^j(A(P))} 1_W = 1_{A(R)}$$ 
and
$$\sum_{R\in \mathcal{P}^j(A(P))} m_K(R)^{1-sp} \leq \Crr{esss}(\Pi_{i\neq 1} \alpha_i/\alpha_1)^{sp}\Crr{de000d}^{(j-k_0(A(P),\mathcal{P}))(1-Dsp)}  m_K(A(P))^{1-sp}.$$
We have $m(A(P))= \alpha_1 \cdots \alpha_D m(P)$.  Let $\Cll{au2}^{-1}d$ and $\Cll{aau}^{-1}\hat{d}$ be  the euclidean metric. Note that (replacing the constants if necessary) we may assume that both $W$ and elements of $\mathcal{D}_{\hat{K}}$ satisfy (\ref{pphe}). 
There is $\Cll{err}$ such that for  every open set  $Q$ that satisfies (\ref{pphe}) we have 
$$\frac{1}{ \Crr{err}} (1/2)^{k_0(Q,\mathcal{D}_{\hat{K}})} \frac{\Crr{au2}}{\Crr{aau}}\alpha  \leq \Crr{de000d}^{k_0(A(Q),\mathcal{P})}\leq  \alpha \frac{\Crr{au2}}{\Crr{aau}} (1/2)^{k_0(Q,\mathcal{D}_{\hat{K}})}\Crr{err},$$
in particular (\ref{v45}) gives
\begin{equation}\label{v455} \sum_{\substack{P\in \hat{\mathcal{F}}(W)\\ k_0(A(P),\mathcal{P})=j    }} m(A(P))^{1-sp} \leq \Cll{abs3h} \Crr{de000d}^{(j-k_0(A(W),\mathcal{P}))(\hat{\eta}-Dsp)}m(A(W))^{1-sp},\end{equation}
so
\begin{align*}&\sum_{j \leq  \ell}\sum_{\substack{P\in \hat{\mathcal{F}}(W)\\ k_0(A(P),\mathcal{P})=j    }}  \sum_{R\in \mathcal{P}^\ell(A(P))} m(R)^{\hat{\eta}-sp}\\
&\leq \sum_{j \leq  \ell} \Crr{esss}(\Pi_{i\neq 1} \frac{\alpha_i}{\alpha_1})^{sp}\Crr{de000d}^{(\ell - j)(\hat{\eta} -Dsp)}  \sum_{\substack{P\in \hat{\mathcal{F}}(W)\\ k_0(A(P),\mathcal{P})=j    }}  m(A(P))^{1-sp} \\
&\leq \Crr{esss}  \Crr{abs3h}(\Pi_{i\neq 1} \frac{\alpha_i}{\alpha_1})^{sp} m(A(W))^{1-sp} \sum_{j \leq \ell} \Crr{de000d}^{(\ell -j)(\hat{\eta}-Dsp)} \Crr{de000d}^{(j -k_0(A(W),\mathcal{P}))(\hat{\eta}-Dsp)}\\
&\leq \Crr{esss}  \Crr{abs3h}(\Pi_{i\neq 1}  \frac{\alpha_i}{\alpha_1})^{sp}  m(A(W))^{1-sp}\Crr{de000d}^{(1-\delta)(\ell -k_0(A(W),\mathcal{P}))(\hat{\eta}-Dsp)} \sum_{j \leq \ell} \Crr{de000d}^{\delta (\ell -j)(\hat{\eta}-Dsp)}\\
&\leq \Crr{essss}  (\Pi_{i\neq 1}  \frac{\alpha_i}{\alpha_1})^{sp} \Crr{de000d}^{(1-\delta)(\ell -k_0(A(W),\mathcal{P}))(\hat{\eta}-Dsp)}m(A(W))^{1-sp}.\end{align*}
and the same inequality holds replacing  $m$ by $m_K$.
\end{proof} 

\begin{proposition}\label{def} Let $(K,m_K,\mathcal{P})$ be  a measure space with a good grid as in Proposition \ref{ant}.  For every  small $\delta$ there is $\Cll{esssss} >0$ such that the following holds.  Let $O$ be an open set and $h\colon O\rightarrow \mathbb{R}^D$ be a $C^1$-diffeomorphism and $O_1\subset O$ be an open subset such that $\overline{O}_1$ is compact. There is $\delta_1 > 0$ such that if  
\begin{itemize}
\item[A.] $W$ is  an  open set satisfying (\ref{pphe}) and (\ref{pphe2}), with $Dsp < \min\{ \hat{\eta},\eta\}$,
\item[B.] $W\subset O_1$ and $diam \ W < \delta_1$, 
\item[C.] $h(W)\subset K$,
\end{itemize}
then $h(W)$ is a $$(1-sp, \Crr{esssss}  (\Pi_{i\neq 1}  \frac{\alpha_i}{\alpha_1})^{sp}, \Crr{de000d}^{(1-\delta)(\hat{\eta}-Dsp)})$$ regular domain in $(K,m,\mathcal{P})$. Here $0< \alpha_1\leq \alpha_2 \leq \cdots \leq \alpha_n$ are such that $\{\alpha_i^2\}$ are the  eigenvalues of $AA^\star$, with $A=D_{x_0}h$ and $x_0$ is an arbitrary element of $W$. The constant $\Crr{esssss}$ depends only on  $\delta$ and the constants in (\ref{ppd}), (\ref{pp2d}) and (\ref{pppd}) for the good grid $\mathcal{P}$ considering $d$ as either the euclidean distance or a multiply of it. 
\end{proposition} 
\begin{proof} Since $\overline{O}_1$ is compact we can write 

$$h(y)-h(x)=D_xh(y-x) + R(x,y)$$
where 
$$\lim_{\substack{|x-y|\rightarrow 0\\ x,y \in I}} \frac{|R(x,y)|}{|x-y|}=0.$$
Since $x\mapsto D_xh$ is continuous in $\overline{O}_1$ and $D_xh$ is invertible for every $x\in \overline{O}_1$, it is easy to see that there there is $\delta_1> 0$ such that 
$$\frac{1}{2}|x-y|\leq | (D_{x_0}h)^{-1} h(y)- (D_{x_0}h)^{-1} h(x)|\leq 2|x-y|$$
for every $x,y  \in B(x_0,\delta_1)$, $x_0\in O_1$.  Apply Proposition  \ref{lip} for $g(x)=(D_{x_0}h)^{-1} \circ h(x)$ and $W$, then apply Proposition \ref{qwe} for $A=D_{x_0}h$ and $g(W)$.
\end{proof}

\section{Regular Potentials}

\subsection{H\"older jacobian} 

\begin{proposition} \label{holderr}  There is $\Cll{uuu}$, that  {\it depends only } on the good grid $\mathcal{P}$, with the following property. 
\begin{itemize}
\item[A.] Suppose that  $\Omega\subset I$ is an interval and the  function $$g\colon \Omega\rightarrow \mathbb{R}$$ satisfies 
$$|g(x)-g(y)|\leq \Cll{hold} d(x,y)^{D(\beta+\epsilon)}$$
for every $x,y \in \Omega$.
Then there if  $W \in \mathcal{P}$ is  such that $W \subset \Omega$  and 
$$\Crr{hold} |W|^{\beta + \epsilon} \leq \sup_W g,$$
then 
$$|g1_W|_{\mathcal{B}^\beta_{p,q}(W,\mathcal{P}_W,\mathcal{A}^{sz}_{p,q})}  \leq  2\Crr{uuu}(\sup_W g)     |W|^{1/p-\beta}.$$
\item[B.] If additionally the function  $g$ is the jacobian of $h$, that is, there
$$m(h(A))=\int_A g \ dm.$$
for every measurable set $A$. If  $W, Q\in \mathcal{P}$ be  such that $W \subset J$, $Q \subset I$ and $h(W)\subset Q$ and 
$$ \Crr{hold} (diam \ h^{-1}(Q) )^{D(\beta+\epsilon)} + \Crr{hold} |W|^{\beta + \epsilon}\leq  \frac{|Q|}{|h^{-1}(Q)|}.$$ then 
$$|g1_W|_{\mathcal{B}^\beta_{p,q}(W,\mathcal{P}_W,\mathcal{A}^{sz}_{p,q})}  
\leq 2\Crr{uuu} \frac{|Q|}{|h^{-1}(Q)|}  |W|^{1/p-\beta}.$$
\end{itemize}
\end{proposition}
\begin{proof}  Let $W, Q\in \mathcal{P}$ be  such that $W \subset J$, $Q \subset I$ and $h(W)\subset Q$. 
The function $\phi\colon I \rightarrow \mathbb{R}$ given by $\phi(x)=0$ if $x \not\in W$, and 
$$\phi(x)= \frac{g(x)}{\Crr{hold} |W|^{\beta + \epsilon} +\sup_W g} |W|^{\beta-1/p} 1_W$$
otherwise, is a $(\beta,\beta+\epsilon,p)$-H\"older  atom supported on $W$. Indeed
$$|\phi|_\infty\leq |W|^{\beta-1/p}$$
and for every $x,y\in W$ 
\begin{align*}
|\phi(x)-\phi(y)|&\leq  \frac{\Crr{hold} }{\Crr{hold} |W|^{\beta + \epsilon} +\sup_W g}   |W|^{\beta-1/p}d(x,y)^{D(\beta+\epsilon)}\\
&\leq  \frac{\Crr{hold} |W|^{\beta + \epsilon}   }{ \Crr{hold} |W|^{\beta + \epsilon} +\sup_W g}   |W|^{\beta-1/p-(\beta+\epsilon)}  d(x,y)^{D(\beta+\epsilon)}\\
&\leq    |W|^{\beta-1/p-(\beta+\epsilon)}  d(x,y)^{D(\beta+\epsilon)}
\end{align*}
This implies  that
$$g1_W=( \Crr{hold} |W|^{\beta + \epsilon} +\sup_W g)|W|^{1/p-\beta}\phi,$$
so
\begin{equation}\label{kkk}
|g1_W|_{\mathcal{B}^\beta_{p,q}(W,\mathcal{P}_W,\mathcal{A}^{sz}_{p,q})}  \leq \Crr{uuu} (\Crr{hold} |W|^{\beta + \epsilon} +\sup_W g) |W|^{1/p-\beta}\end{equation}
where $\Crr{uuu}$ {\it depends only } on the good grid $\mathcal{P}$.\\

\noindent {\bf Proof of A.} We  have 
\begin{align*}
|g1_W|_{\mathcal{B}^\beta_{p,q}(W,\mathcal{P}_W,\mathcal{A}^{sz}_{p,q})}  &\leq \Crr{uuu} (\Crr{hold} |W|^{\beta + \epsilon} +\sup_W g) |W|^{1/p-\beta}\\
&\leq 2\Crr{uuu}(\sup_W g)   |W|^{1/p-\beta}.
\end{align*}

\noindent {\bf Proof of B.} Note that for every $x\in h^{-1}(Q)$
\begin{align*}
g(x)&=\frac{1}{|h^{-1}(Q)|}\int_{h^{-1}(Q)} g(x) \ dm(y)\\
&= \frac{1}{|h^{-1}(Q)|}\int_{h^{-1}(Q)} g(y) \ dm(y)+   \frac{1}{|h^{-1}(Q)|}\int_{h^{-1}(Q)} g(x)-g(y) \ dm(y)\\
&\leq   \frac{|Q|}{|h^{-1}(Q)|}+ \Crr{hold} (diam \ h^{-1}(Q) )^{D(\beta+\epsilon)} .
\end{align*}
so in particular 
$$\sup_W g \leq   \frac{|Q|}{|h^{-1}(Q)|}+  \Crr{hold} (diam \ h^{-1}(Q) )^{D(\beta+\epsilon)} .$$
So if 
$$ \Crr{hold} (diam \ h^{-1}(Q) )^{D(\beta+\epsilon)} + \Crr{hold} |W|^{\beta + \epsilon}\leq  \frac{|Q|}{|h^{-1}(Q)|}.$$
Then
\begin{align*}
|g1_W|_{\mathcal{B}^\beta_{p,q}(W,\mathcal{P}_W,\mathcal{A}^{sz}_{p,q})}  &\leq \Crr{uuu} (\Crr{hold} |W|^{\beta + \epsilon} +\sup_W g) |W|^{1/p-\beta}\\
&\leq   \Crr{uuu}\big(  \frac{|Q|}{|h^{-1}(Q)|}+   \Crr{hold} (diam \ h^{-1}(Q) )^{D(\beta+\epsilon)}  + \Crr{hold} |W|^{\beta + \epsilon}\big)  |W|^{1/p-\beta}\\
&\leq 2\Crr{uuu} \frac{|Q|}{|h^{-1}(Q)|}  |W|^{1/p-\beta}.
\end{align*}
\end{proof}

\subsection{Non-flat critical points on an interval} Here we consider the interval $[0,1]$ with the dyadic grid and the Lebesgue measure.   Define  $h\colon [0,1]\rightarrow [0,1]$ as $h(x)=x^\alpha$, with $\alpha >   1$ and
$$g(x)=h'(x)=\alpha x^{\alpha-1}.$$

To simplify the notation we denote $\gamma=\alpha-1 \in (0,\infty)$. 

\begin{lemma}\label{le1} If  $\gamma \in (0,1)$ then $g$ is $\gamma$-H\"older continuous, that is, there is $K_\gamma$ such that
$$|g(x)-g(y)|\leq K_\gamma |x-y|^\gamma.$$
\end{lemma} 
\begin{proof} Consider  $0\leq  y < x$ and 
$$T(y,x)= \frac{x^\gamma- y^\gamma}{(x-y)^\gamma}.$$
Of course $T(0,x)=1$ for every $x> 0$. Note that $T(\lambda x, \lambda y)=T(x,y)$ for every $\lambda > 0$, so  $\sup_{0<  y < x} T(x,y)= \sup_{1< x} T(1,x)$. It is easy to see that
$$\lim_{x\rightarrow 1+} T(1,x)= 0 \ and \  \lim_{x\rightarrow+\infty} T(1,x)=1,$$
so    we can take $K_\gamma= (\gamma+1)\sup_{0\leq  y < x} T(x,y) < \infty.$
\end{proof} 

\begin{lemma} \label{le2}Suppose $\gamma > 0$ and  $b \in  (c,d]\subset [0,\infty).$ Then
$$\sup_{c < b\leq d}  \frac{b^{\gamma}  (d-c)}{d^{1+\gamma}-c^{1+\gamma}} \leq 1.$$

\end{lemma} 
\begin{proof} For every $b,c,d$ satisfying $0\leq c < b\leq d$  define
$$T(c,b,d)=  \frac{b^{\gamma}  (d-c)}{d^{1+\gamma}-c^{1+\gamma}}$$
If $c=0$ then    
$$T(0,b,d)=  \frac{b^\gamma}{d^\gamma}\leq 1.$$
Note that for every $\lambda > 0$ and $c> 0$ we have
$$T(\lambda c, \lambda b, \lambda d)= T(c,b,d),$$
so we have
$$\sup_{0< c < b\leq d}  \frac{b^{\gamma}  (d-c)}{d^{1+\gamma}-c^{1+\gamma}}=  \sup_{1 < b\leq d}  \frac{b^{\gamma}  (d-1)}{d^{1+\gamma}-1} \leq  \sup_{1< b\leq d}  \frac{d^{1+\gamma} -d^\gamma}{d^{1+\gamma}-1} \leq 1.$$
\end{proof} 

\begin{proposition} \label{lorenzz} Suppose that $\beta < \min\{1, \gamma\}$. Then there is $\Crr{novac}$ such that  the following holds. Let $W$ and $Q$ be intervals such that  $W\subset h^{-1}(Q)$. Then 
$$|g1_W|_{\mathcal{B}^\beta_{p,q}(W,\mathcal{P}_W,\mathcal{A}^{sz}_{p,q})}  \leq    \Crr{novac} |W|^{1/p-\beta}  \frac{|Q|}{|h^{-1}Q|}.$$

\end{proposition} 
\begin{proof} We consider two cases:

\noindent {\bf Case A.} Suppose $\gamma < 1$. Choose $\delta > 0$  such that $\beta + 2\delta=\gamma=\alpha -1$.  Then by Lemma \ref{le1} 
$$|g(x)-g(y)|\leq K_\gamma |x-y|^{\beta + 2 \delta}.$$
Let  $W=[a,b]$ and $h^{-1}Q=[c,d]$. Define
$$\phi(x)=  \frac{|W|^{\beta - 1/p}}{(K_\gamma +\gamma+1) b^{\gamma}}  g\cdot 1_W.$$
We claim that $\phi$ is a $(\beta,\beta+\delta,p)$-H\"older  atom supported on $W$.  Indeed
$$|\phi|_\infty \leq \frac{|W|^{\beta - 1/p} (\gamma+1)b^{\gamma}}{(K_\gamma +\gamma +1) b^{\beta + 2\delta}}\leq |W|^{\beta - 1/p}.$$
and for every $x,y \in W$ we have
\begin{align*}
|\phi(x)-\phi(y)|&\leq \frac{K_\gamma |W|^{\beta - 1/p }}{(K_\gamma +\gamma+1) b^{\beta + 2\delta}} |x-y|^{\beta + 2 \delta} \\
&\leq \frac{ |W|^{\beta - 1/p } b^\delta}{ b^{\beta + 2\delta}} |x-y|^{\beta +  \delta} \\
&\leq \frac{ |W|^{\beta - 1/p }}{ b^{\beta + \delta}} |x-y|^{\beta +  \delta} \\
&\leq |W|^{\beta - 1/p-(\beta+\delta)}  |x-y|^{\beta +  \delta}.
\end{align*} 
Consequently by Lemma \ref{le2}
\begin{align*}
|g1_W|_{\mathcal{B}^\beta_{p,q}(W,\mathcal{P}_W,\mathcal{A}^{sz}_{p,q})} & \leq   \Cll{novac} |W|^{1/p-\beta} b^{\gamma}\\
&\leq \Crr{novac} |W|^{1/p-\beta} \frac{d^{1+\gamma}-c^{1+\gamma}}{d-c}\\
&\leq \Crr{novac} |W|^{1/p-\beta}  \frac{|Q|}{|h^{-1}Q|}.
\end{align*} 

\noindent {\bf Case B.} Suppose $\gamma \geq 1$. Choose $\delta > 0$ such that $\beta+\delta < 1$. If $0\leq x< y$  we have 
$$|g(x)-g(y)|\leq  (1+\gamma)\gamma y^{\gamma-1}|x-y|\leq  (1+\gamma)\gamma y^{\gamma-(\beta+\delta)}|x-y|^{\beta+\delta}.$$
Define
$$\phi(x)=  \frac{|W|^{\beta - 1/p}}{(\gamma+1)^2 b^{\gamma}}  g\cdot 1_W.$$
We claim that $\phi$ is a $(\beta,\beta+\delta,p)$-H\"older  atom supported on $W$.  Indeed
$$|\phi|_\infty \leq \frac{|W|^{\beta - 1/p} (\gamma+1)b^{\gamma}}{(\gamma +1)^2 b^{\gamma}}\leq |W|^{\beta - 1/p}.$$
and for every $x,y \in W=[a,b]$, with  $0\leq x < y$, we have
\begin{align*}
|\phi(x)-\phi(y)|&\leq \frac{  (1+\gamma)\gamma y^{\gamma-(\beta+\delta)} |W|^{\beta - 1/p }}{(\gamma+1)^2 b^\gamma} |x-y|^{\beta + \delta} \\
&\leq \frac{ |W|^{\beta - 1/p }}{ b^{\beta+\delta}} |x-y|^{\beta +  \delta} \\
&\leq \frac{ |W|^{\beta - 1/p }}{ b^{\beta + \delta}} |x-y|^{\beta +  \delta} \\
&\leq |W|^{\beta - 1/p-(\beta+\delta)}  |x-y|^{\beta +  \delta}.
\end{align*} 
and now we can complete the proof exactly as in Case A. 
\end{proof}

\subsection{$1/\beta$-bounded  variation potentials} Here we consider the interval $[0,1]$ with the dyadic grid and the Lebesgue measure. 

\begin{proposition} \label{bvp}  There is $\Cll{uuug}$, that does not depend  on $h$, with the following property. 
\begin{itemize}
\item[A.] Suppose that  the  function $$g\colon \Omega\rightarrow \mathbb{R}$$  has finite $1/(\beta+\epsilon)$-bounded variation. Then there is a finite partition by intervals  $\{\Omega_1', \dots, \Omega_n'\}$ of $\Omega'$  such that if  $W \in \mathcal{P}$ and $W\subset h^{-1}(\Omega_i')$, for some $i$, then 
$$|g1_W|_{\mathcal{B}^\beta_{p,q}(W,\mathcal{P}_W,\mathcal{A}^{sz}_{p,q})}  \leq  \Crr{uuug} (\sup_W g)   |W|^{1/p-\beta}.$$
\item[B.] Suppose  additionally  that the function  $g$ is the jacobian of $h$, that is, 
$$m(h(A))=\int_A g \ dm.$$
for every measurable set $A$. Then   if  $W, Q\in \mathcal{P}$ be  such that $W \subset \Omega$, $Q \subset \Omega'_i$ for some $i$, and $h(W)\subset Q$ then 
$$|g1_W|_{\mathcal{B}^\beta_{p,q}(W,\mathcal{P}_W,\mathcal{A}^{sz}_{p,q})}  
\leq 3\Crr{uuug} \frac{|Q|}{|h^{-1}(Q)|}  |W|^{1/p-\beta}.$$
\end{itemize}
\end{proposition}
\begin{proof}   Let $W\in \mathcal{P}$, with $W\subset \Omega$.  Define the function 
$$\phi = \frac{|W|^{s-1/p}}{var_{1/\beta}(g,W)+ \sup_W  g}  g 1_W.$$
Of course $|\phi|_\infty \leq |W|^{s-1/p}$ and $var_{1/\beta}(\phi, W)\leq  |W|^{s-1/p} .$ So $\phi$ is a $\mathcal{A}^{bv}_{s, p,\beta}(W)$-atom, consequently 
$$
|g1_W|_{\mathcal{B}^\beta_{p,q}(W,\mathcal{P}_W,\mathcal{A}^{sz}_{p,q})}  \leq \Crr{uuug} ( var_{1/\beta}(g,W)+ \sup_W  g)|W|^{1/p-s}.  $$
Since the $1/(\beta+\epsilon)$-bounded variation of $g$  is finite, we can find a finite partition by intervals  $\{\Omega_1', \dots, \Omega_n'\}$ of $\Omega'$ such that  for each $i$ 
$$var_{1/\beta}(g,h^{-1}(\Omega_i'))=\sup \big( \sum_{k=0}^{n}  |g(x_{k+1})-g(x_k)|^{1/\beta}\big)^{\beta} < \min \{ \sup_\Omega g, \frac{1}{\Crr{lm}}  \}.$$
where the sup runs over all possible finite sequences $x_k < x_{k+1}$ and $x_k \in int \ h^{-1}(\Omega_i')$. \\

\noindent Let $W\subset h^{-1}(Q)$, with $Q\subset \Omega_i'$ for some $i$. \\

\noindent {\bf Proof of A.} We have  
$$|g1_W|_{\mathcal{B}^\beta_{p,q}(W,\mathcal{P}_W,\mathcal{A}^{sz}_{p,q})}  \leq 2 \Crr{uuug} ( \sup_\Omega  g)|W|^{1/p-s}.  $$
\noindent {\bf Proof of B.} Note that for every $x\in h^{-1}(Q)$
\begin{align*}
g(x)&=\frac{1}{|h^{-1}(Q)|}\int_{h^{-1}(Q)} g(x) \ dm(y)\\
&= \frac{1}{|h^{-1}(Q)|}\int_{h^{-1}(Q)} g(y) \ dm(y)+   \frac{1}{|h^{-1}(Q)|}\int_{h^{-1}(Q)} g(x)-g(y) \ dm(y)\\
&\leq   \frac{|Q|}{|h^{-1}(Q)|}+ \frac{1}{\Crr{lm}}\leq 2  \frac{|Q|}{|h^{-1}(Q)|},
\end{align*}
so 
\begin{align*}
|g1_W|_{\mathcal{B}^\beta_{p,q}(W,\mathcal{P}_W,\mathcal{A}^{sz}_{p,q})}  &\leq \Crr{uuug} ( var_{1/\beta}(g,W)+ \sup_W  g)|W|^{1/p-s}\\
&\leq \Crr{uuug} ( \frac{1}{\Crr{lm}}+2  \frac{|Q|}{|h^{-1}(Q)|})|W|^{1/p-s}\\
&\leq 3\Crr{uuug}  \frac{|Q|}{|h^{-1}(Q)|}|W|^{1/p-s}.
\end{align*}
\end{proof}

\section{Some strongly regular domains in $\mathbb{R}^D$}

Let $(I,m,\mathcal{P})$ be a measure space with a good grid. Recall that a  subset $\Omega\subset I$ is a {\bf $(\alpha,\Cll{rpf},t )$-strongly  regular domain (see  \cite{smania-besov})}  if  for every   $Q\in \mathcal{P}^i$,  $i\geq t$  and $k\geq k_0(Q\cap \Omega)$ there are families  $\mathcal{F}^k(Q\cap \Omega) \subset \mathcal{P}^k$  satisfying 
\begin{itemize}
\item[A.] We have $Q\cap \Omega = \cup_{k\geq k_0(Q\cap \Omega)} \cup_{P\in \mathcal{F}^k(Q\cap \Omega)} P$.
\item[B.] If $P,W \in \cup_{k} \mathcal{F}^k(Q\cap \Omega)$ and $P\neq W$ then $P\cap W=\emptyset$. 
\item[C.] We have
\begin{equation} \sum_{P\in \mathcal{F}^k(Q\cap \Omega)} |P|^{\alpha}\leq \Crr{rpf}  |Q|^{\alpha}.\end{equation} 
\end{itemize}

 It is easy to prove that 

\begin{proposition} \label{d-1} There is $\Cll{dim} > 0$ such that the following holds. Let 
$K= \cup_i \hat{M}_i,$
where $M_i$ is a compact  $(D-1)$-dimensional  $C^1$-manifold with boundary embedded in  $\mathbb{R}^D$. Moreover assume that  for every $x \in \partial K$ there is $r_x > 0$ such that
$$\#\{ i \colon B(x,r) \cap  M_i \neq \emptyset \}\leq N$$
for every $r < r_x$. Then there is $r_0$ such that for every $x \in K$ and $r \in (0,r_0)$ we have
\begin{equation} \label{srd} \frac{1}{\Crr{dim}} r^{D-1} \leq m_{D-1}(B(x,r)\cap K)\leq \Crr{dim} N r^{D-1}.\end{equation} Here $m_{D-1}$ denotes the $(D-1)$-dimensional Hausdorff measure. We emphasize that $\Crr{dim}$ does not depend  on $K$. 
\end{proposition}

A {\bf $N$-good $C^r$ domain}  $P$ in $\mathbb{R}^D$  is an open subset of $\mathbb{R}^D$  for which $\partial P$ is compact  and there a finite number of $(D-1)$-dimensional $C^m$ manifolds with boundary $M_i$ embedded in $\mathbb{R}^D$, with $m\geq 1$, such that 
$$\partial P \subset \cup_{i\leq  k} M_i$$
and  such that for every $x \in \partial P$ there is $r_x > 0$ satisfying 
$$\#\{ i \colon B(x,r) \cap  M_i \neq \emptyset \}\leq N$$
for every $r < r_x$. A simple example of a $N$-good $C^r$ domain is a convex set defined as the intersection of a finite number of half-spaces in $\mathbb{R}^D$. 

We say that $N$-good $C^1$ domain has a {\it regular Whitney stratification}  if we can choose the the manifolds with boundary $M_i$ such that $\cup_i M_i$ has a Whitney regular $C^1$  stratification. We will not need this property in this section, however it will be useful to study generic piecewise expanding maps on $\mathbb{R}^D$  in Section \ref{pie} following an argument similar to  Cowieson \cite{cowieson}.

\begin{corollary}\label{Ngood} For every $N$ there exists $\Cll{srd2}$ such that the following holds. For every $N$-good $C^r$ domain  $P$ in $\mathbb{R}^D$ there exists $t$ such that $P$ is a  $(1-\frac{1}{D},\Crr{srd2},t)$-strongly regular domain.
\end{corollary} 
\begin{proof}  This follows from Proposition \ref{d-1} and  Proposition \ref{homo-quasi} in \cite{smania-homo}. 
\end{proof}

\begin{remark} The class of strongly regular domains in $\mathbb{R}^D$ is much wider than the class of $N$-good $C^1$ domains. For instance a domain whose boundary  $K$ is  a cone with circular base is a strongly regular domain since it satisfies (\ref{srd}) if we replace   $\Crr{dim}$ for some appropriated constant. Indeed certain domains with fractal boundary are also  strongly regular domains (see Remark \ref{homo-juliaex} in \cite{smania-homo}). However the advantage of $N$-good $C^1$ domains is that 
$\Crr{dim}$ does not depend on the particular domain we are considering, that is very handy for estimate the essential spectral radius of transfer operators.
\end{remark}

\vspace{1cm}
\centerline{ \bf III. THE TOY MODEL.}
\addcontentsline{toc}{chapter}{\bf III. THE TOY MODEL.}
\vspace{1cm}

\section{Linear expanding map acting on  the circle} 

  Let $\ell\in \mathbb{N}\setminus \{0,1\}$  and define  $f_\ell\colon [0,1]\rightarrow [0,1]$ as $$f_\ell(x)=\ell x \mod 1.$$
Let $\mathcal{D}_\ell^k$ be the partition of $[0,1]$ in $\ell^k$ intervals with same size.  Then $\mathcal{D}_\ell=(\mathcal{D}^k_\ell)_{k\in \mathbb{N}}$ is a good grid.   The  map $f_\ell$ is a classic example of  expanding map. The Lebesgue measure $m$ on $[0,1]$ is an invariant probability for $f_\ell$. Our goal is to prove the Lasota-Yorke inequality  for $f_\ell$ in the space $\mathcal{B}^s_{1,1}([0,1],m,\mathcal{D}_\ell)$, with $s\in (0,1)$.  That is, we will find $j> 0$, $C > 0$ and $\lambda \in (0,1)$ such that 
$$|\mathcal{L}_\ell^j(\phi)|_{\mathcal{B}^s_{1,1}}\leq  C |\phi|_1 + \lambda  |\phi|_{\mathcal{B}^s_{1,1}}$$
for every $\phi \in \mathcal{B}^s_{1,1}([0,1],m,\mathcal{D}_\ell)$.

 This is not a new result, however the we with this  example since its simplicity allows us to give a very detailed and yet short proof of the Lasota-Yorke inequality that illustrate the  methods of this paper. 

Consider the Ruelle-Perron-Frobenious operador of $f_{\ell^j}$
$$(\mathcal{L}_{\ell^j} \psi)(x)=\sum_{i=0}^{\ell^j-1}  \frac{1}{\ell^j} \psi(\frac{x}{\ell^j}+ \frac{i}{\ell^j}).$$

We have that $\mathcal{D}_{\ell^j}^1$ is a markovian partition for $f_\ell$. Moreover if $P\in \mathcal{D}^k_\ell$ and  $k\geq j$, then  $Q=f_{\ell^j}(P)\in  \mathcal{D}^{k-j}_\ell$
and 
 $$\mathcal{L}_{\ell^j}  (a_{P})=\mathcal{L}_{\ell^j} (|P|^{s-1}1_P)=\frac{1}{\ell^j} |P|^{s-1}1_{Q}=\frac{1}{\ell^{sj}} |\ell^j|^{s-1} | P|^{s-1}1_{Q}=\frac{1}{\ell^{sj}} a_Q.$$
And if   $P\in \mathcal{D}^k_\ell$ and  $k < j$ then $P$ is an union of $\ell^{j-k}$ elements of $\mathcal{D}^j_\ell$, so
 $$\mathcal{L}_{\ell^j}  (a_{P})=\mathcal{L}_{\ell^j} (|P|^{s-1}\sum_{\substack{P\in \mathcal{D}^j_\ell\\Q\subset P }}1_Q)=\frac{\ell^{j-k}}{\ell^j} |P|^{s-1}= \ell^{-ks}a_{[0,1]}.$$

 By Arbieto and S. \cite[Proposition \ref{trans-canrep2}]{smania-transfer} there is $\Cll[name]{GC}$ such that  for every $P\in \mathcal{D}_\ell$ there are linear functionals $\phi\mapsto k^\phi_P$  in $(L^1)^\star$  such that every $\phi\in \mathcal{B}^s_{1,1}([0,1],m, \mathcal{D}_\ell)$  has a $\mathcal{B}^s_{1,1}$-representation 
 
 $$\phi= \sum_k  \sum_{P\in \mathcal{D}_\ell^k} k^\phi_P a_P$$ satisfying 
 $$\sum_k \sum_{P\in \mathcal{D}^k_\ell} |k_P^\phi|  \leq \Crr{GC} |\phi|_{\mathcal{B}^s_{1,1}}.$$
In particular
$$\mathcal{L}_{\ell^j}(\phi)=    \big( \sum_{k<  j}  \frac{1}{\ell^{ks}} \sum_{P\in \mathcal{D}_\ell^k} k^\phi_{P}  \big)a_{[0,1]}   + \frac{1}{\ell^{sj}} \sum_{k\geq j}  \sum_{P\in \mathcal{D}_\ell^k} k^\phi_P a_{f_\ell{P}}$$
and
\begin{align*} |\mathcal{L}_{\ell}^j(\phi)|_{\mathcal{B}^s_{1,1}}= |\mathcal{L}_{\ell^j}(\phi)|_{\mathcal{B}^s_{1,1}}&\leq    \sum_{k<  j}  \frac{1}{\ell^{ks}} \sum_{P\in \mathcal{D}_\ell^k} |k^\phi_{P}|   + \frac{1}{\ell^{sj}} \sum_{k\geq j}  \sum_{P\in \mathcal{D}_\ell^k} |k^\phi_P|\\
&\leq \Big(   \sum_{k<  j}  \frac{1}{\ell^{ks}} \sum_{P\in \mathcal{D}_\ell^k} |k_{P}|_{(L^1)^\star} \Big)   |\phi|_1+ \frac{\Crr{GC}}{\ell^{sj}} |\phi|_{\mathcal{B}^s_{1,1}}. \end{align*}
We can choose $j$ large enough to have 
$$\lambda = \frac{\Crr{GC}}{\ell^{js}} < 1$$
so we obtain the Lasota-Yorke inequality.

\vspace{1cm}
\centerline{ \bf IV. THE BESTIARY.}
\addcontentsline{toc}{chapter}{\bf  IV. THE BESTIARY.}
\vspace{1cm}

\section{Markovian expanding maps}
\label{markov}

Markovian maps arise in the very beginning of the study of the metric theory of expanding maps, as the Gauss map and the linear expanding maps on the circle.  The work of Ruelle \cite{ruelle1d} deals with the one-sided shift, another example of such maps. 

Sinai \cite{markovp}  constructed Markov partitions for expanding maps on manifolds \cite{tf}  and hyperbolic diffeomorphisms, See   Bowen \cite{bowen} and  Parry and Pollicott \cite{pp} for more details. 

Let $(I,m)$ be a probability space. Suppose that there is collection of subsets $\{I_1,\dots,I_n\}$ of $I$,with $n\geq 2$, a transformation 
$$f\colon \cup_i I_i  \rightarrow I$$
satisfying 
\begin{itemize}
\item[A.] We have $I_i\cap I_j=\emptyset$ for every $i\neq j$.
\item[B.]  $m(I_i)  > 0$ for every $i$  and $m(I\setminus \cup_i I_i)=0$.
\item[C.] The set $f(I_i)$ is measurable, and $f\colon I_i \rightarrow f(I_i)$ is a bijection with measurable inverse $h_i\colon f(I_i)\rightarrow I_i$ and  Jacobian $w$, that is, for every  measurable set $A\subset I_i$
$$m(f(A))=\int_A w \ dm.$$   
Moreover $\inf w > 1$.
\item[D.] For every $i$  the set $f(I_i)$ is a union of at least two elements of $\{I_1,\dots,I_n\}$. 
\end{itemize} 
Then we can define a sequence of partitions of $I$ recursively as  $\mathcal{P}^0=  \{I\}$, $\mathcal{P}^1=\{I_1,\dots,I_n\}$ and for $k\geq 1$

$$\mathcal{P}^{k+1}= \{ h_i(Q), \text {where }   Q\in \mathcal{P}^{k}  \text {and } Q\subset f(I_i)   \}.$$

Note that $\mathcal{P}=(\mathcal{P}^k)_k$ is a nested sequence of partitions and every element of $\mathcal{P}^k$ has at least two children.  Of course the $j$-th iteration $f^j$ of $f$ has similar properties. Indeed for every $P\in \mathcal{P}^j$ we have that 
$$f^j\colon P \rightarrow f^j(P)$$ 
is a bijection with Jacobian
$$w_P(x)=\Pi_{i=0}^{j-1}w(f^i(x))$$
and a measurable inverse, denoted $h_P$. Moreover $$f^j(P) \in \{f(I_1),\dots,f(I_n)\}.$$
Now assume additionally 
\begin{itemize}
\item[E.] (Bowen Condition) We have
$$\sup_k     \sup_{P\in \mathcal{P}^k} \sup_{x,y \in P}  |\sum_{i=0}^{k-1}\ln w(f^ix)-\sum_{i=0}^{k-1} \ln  w(f^iy)|      < \infty$$
and there is $\Cll{sd}> 0$ such that 
$$\frac{1}{\Crr{sd}}\leq     w(x)\leq \Crr{sd}.$$
for every $x$.
\end{itemize} 
Define  for  each $P\in \mathcal{P}^k$ the function $w_P\colon P\rightarrow \mathbb{R}^\star_+$ as
$$w_P(x)= \Pi_{i=0}^{k-1}w(f^i(x)).$$
It easily follows that there is $\Cll{sd2}> 0$ such that for every $x,P, k$ satisfying $x\in P\in \mathcal{P}^k$
$$ \frac{1}{\Crr{sd2} }  w_P(x) \leq  \frac{1}{ |P|}\leq \Crr{sd2} w_P(x).$$
and $\mathcal{P}$ is a good grid.   We can define a metric in $I$  as
$$d(x,y)=\inf \{|Q|\colon   \  x,y \in Q\in \cup_k \mathcal{P}^k\}.$$
Then $(I,m, d)$ is an Ahlfors-regular metric space ($D=1$). Note that every inverse branch $h_P$, with $P\in \mathcal{P}^j$  is bi-Lipchitz and satisfies (\ref{mc}) for some $\Crr{lm}(j)$ that may depend on $j$. Now assume additionally 
\begin{itemize}
\item[F.] There is some $C$ such  that $w$ satisfies
$$|w(x)-w(y)|\leq C d(x,y)^{\beta+\epsilon}$$
for every $x,y \in I_i$, $i\leq n$.
\end{itemize}

\begin{theorem} \label{marko} Let $f$ be a Markovian map as above. Then
$$r_{ess}(\Phi,\mathcal{B}^s_{p,q}(I,m,\mathcal{P}))\leq \big( \liminf_{j\rightarrow \infty} \big( \sum_{P\in   \mathcal{P}^j}   |P|^{1+sp'}\big)^{1/j}\big)^{1/p'}.$$
In particular $r_{ess}(\Phi)\leq (\inf w)^{-s}< 1.$
\end{theorem}
\begin{proof}
We have that $$g_P(x)=  \frac{1}{w_P(h_P(x))}$$
is the Jacobian of $h_P$.   Using the usual bounded distortion argument one can show that there is $\Cll{lipp}$ such that 
$$|\ln g_P(x) - \ln g_P(y)|\leq \Crr{lipp}d(x,y)^{\beta+\epsilon}$$
for every $x,y \in f^k(P)$, $P\in \mathcal{P}^k$, $k\in \mathbb{N}$.  In particular there is $\delta > 0$ such that if $d(x,y)< \delta$ then
$$\frac{1}{2}\leq   \frac{g_P(x)}{g_P(y)}\leq 2.$$
In particular if $Q\subset P$, $x \in h^{-1}_{P}(Q)$ and $diam \ h^{-1}_{P}(Q) < \delta$ then
$$\frac{1}{2}g_P(x) \leq \frac{|Q|}{|h^{-1}_{P}(Q)|}\leq 2 g_P(x).$$ 
Moreover  the Mean Value Theorem gives
\begin{equation} \label{estt} | g_P(x) -  g_P(y)|\leq 2g(x)\Crr{lipp}d(x,y)^{\beta+\epsilon}.\end{equation} 
Denote $\Crr{hold}(P,x)= 2\Crr{lipp}g(x).$ Reduce $\delta$ if necessary such that  $4\Crr{lipp}\delta^{\beta+\epsilon}< 1/2$. Then if $W\subset h^{-1}_{P}(Q)$  we have 
\begin{align*}
 &\Crr{hold}(P,x) (diam \ h^{-1}_P(Q) )^{\beta+\epsilon} + \Crr{hold}(P,x) |W|^{\beta + \epsilon}\\
 &\leq   4\Crr{lipp}g(x) \delta^{\beta+\epsilon}  \leq \frac{1}{2} g_P(x) \leq \frac{|Q|}{|h^{-1}(Q)|}.
 \end{align*}
 Choose $i_0\geq 2$ such that $diam \ R < \delta$ for every $R\in \mathcal{P}^{i_0}$. Define $\Lambda_j=\mathcal{P}^{j+i_0}$. Given $R\in \Lambda_j$, let $\tilde{h}_R$ be the restriction of $h_P$ to $f^j(R)$, where $P\in \mathcal{P}^{j}$ satisfies $R\subset P$. Note  that $\tilde{h}_R^{-1}(R)$ is a union of elements of  $\mathcal{P}^{i_0}$. In an analogous way, define $\tilde{g}_R$ as the restriction of $g_P$ to $f^j(R)$.  Then the Ruelle-Perron-Frobenious of $f^j$ can be written as
 $$(\Phi^j\psi)(x)=\sum_{R\in \Lambda_j}   \tilde{g}_R(x)\psi(\tilde{h}_R(x)).$$
By Proposition \ref{holderr}.B There is $\Crr{uuu}$, that depends only on the good grid $\mathcal{P}$ (in particular {\it it  does not depend on $j$}) with the following property. If  $W, Q\in \mathcal{P}$  are   such that $W \subset f^j(R)$, $Q \subset R\in \Lambda_j$ and $h_i(W)\subset Q$ then 
$$|\tilde{g}_R1_W|_{\mathcal{B}^\beta_{p,q}(W,\mathcal{P}_W,\mathcal{A}^{sz}_{p,q})}  
\leq 2\Crr{uuu} \frac{|Q|}{|\tilde{h}_R^{-1}(Q)|}  |W|^{1/p-\beta}.$$
Due the Bowen condition we have that 
$$|\tilde{g}_R1_W|_{\mathcal{B}^\beta_{p,q}(W,\mathcal{P}_W,\mathcal{A}^{sz}_{p,q})}  
\leq \Cll[name]{DRP}(R) \Big(\frac{|Q|}{|\tilde{h}_R^{-1}(Q)|}\Big)^{1/p-s+\epsilon}   |W|^{1/p-\beta},$$
where
$$ \Crr{DRP}(R) =  \Cll{auxxs} \Big(\frac{|R|}{|\tilde{h}_R^{-1}(R)|}\Big)^{1-(1/p-s+\epsilon)},$$
where $\Crr{auxxs}$ does not depend on $R\in \Lambda_j$ and $j$.

Of course $R\in \Lambda_j$ is a $(1-\beta p,1,j)$-strongly regular domain. Moreover for every $Q\in \mathcal{P}^k$, with $k\geq j$ we have that 
\begin{equation}\label{num}\# \{R\in \Lambda_j \colon R\cap Q\neq \emptyset    \}\leq 1.\end{equation}
Furthermore  for every $Q\in \mathcal{P}^k$ satisfying $Q\subset R \in \Lambda_j$ we have that $\tilde{h}_R^{-1}(Q)=f^j(Q)$ is an element of $\mathcal{P}^{k-j}$. In particular  we have that $\tilde{h}_R^{-1}(Q)$ is a $(1-sp,\Cll[name]{DGD1},\Cll[namec]{DGD2})$-regular domain, where $\Crr{DGD1}=1$ and {\it  $\Crr{DGD2} \in (0,1)$ can be chosen so close to $0$ as we want to.}

Finally note that if $a_R=j$ for $R\in \Lambda_j$  then
$$ |k_0(Q)-k_0(\tilde{h}^{-1}_R(Q))|=a_R.$$
and  (again due Bowen condition)  there is $\Cll[name]{DC1}$  such that taking 
$$\Cll[namec]{DC2}(R)=\Big(\frac{|R|}{|\tilde{h}^{-1}_R(R)|}\Big)^{1/j}$$
we have 
\begin{equation} \label{koebe} \frac{|Q|}{|h^{-1}_P(Q)|}  \leq \Crr{DC1}( \Crr{DC2}(R))^{ |k_0(Q)-k_0(\tilde{h}_R^{-1}(Q))|}.\end{equation}
for every $Q\in \mathcal{P}$ satisfying $Q\subset R$.  If we choose $\Crr{DGD2}$ small enough we obtain
\begin{equation}\label{fg} \Cll[namec]{DRS2}(R)=\max \{ (\Crr{DC2}(R))^{\epsilon},\Crr{DGD2}^{1/p}\}=(\Crr{DC2}(R))^{\epsilon}  < 1.\end{equation}
Denote $\Crr{DRS2}=\sup_{R\in \cup_j \Lambda_j} \Crr{DRS2}(R)$. Note that  $\Crr{DRS2}\leq (\inf w)^{-\epsilon} < 1$.
 Choose $\gamma_{_{DRS3}}\in (0,1)$. So
\begin{align*}
\Theta_R&=  \Crr{DC1}^{\epsilon} \Crr{DRP}(R)  \Crr{DGD1}^{1/p}  ( \Crr{DRS2}(R))^{a_R (1-\gamma_{_{DRS3}})} \\
&=\Crr{DC1}^{\epsilon}  \Crr{auxx} \Big(\frac{|R|}{|\tilde{h}_R^{-1}(R)|}\Big)^{1-(1/p-s+\epsilon)}  \Crr{DGD1}^{1/p}  \Big(\frac{|R|}{|\tilde{h}^{-1}_R(R)|}\Big)^{\epsilon(1-\gamma_{_{DRS3}})}\\
&= \Crr{DC1}^{\epsilon}  \Cll{auxxx}  \Big(\frac{|R|}{|\tilde{h}^{-1}_R(R)|}\Big)^{1-(1/p-s)-\epsilon\gamma_{_{DRS3}}}
\end{align*}
By Proposition 11.2, the Ruelle-Perron-Frobenious $\Phi^j$ has a $(\Cll[name]{GSR}\Cll{f2}(j),\Crr{GSR}\Cll{f1}(j),\gamma_{_{DRS3}})$-essential slicing, for some constant $\Crr{f2}(j)$ and 
$$\Crr{f1}(j)=  \big( \sum_{R \in \Lambda_j}  \Theta_R^{p'}   \big)^{1/p'} $$
Since there is $C > 0$ such that $|\tilde{h}^{-1}_R(R)|\geq C$ for every $R\in \Lambda_j$ we have that
$$\Crr{f1}(j)\leq   \Crr{DC1}^{\epsilon}  \Crr{auxxx}  \Cll{xx}  \big(\sum_{R\in \mathcal{P}^{j+i_0}}   |R|^{1+sp' -\epsilon\gamma_{_{DRS3}}p'}\big)^{1/p'}  $$ 
Corollary 10.1 in \cite{smania-transfer} tell us that we can write $\Phi^j= K_j+R_j$, whre $K_j$ is a finite-rank operator and 
\begin{align*} |R_j|&\leq \frac{2}{1-\Crr{DRS2}^{\gamma_{_{DRS3}}}}  \Cll[name]{GBS} \Crr{f1}(j) \Cll[name]{GC}\\
& \leq  \frac{2}{1-((\inf w)^{-\epsilon})^{\gamma_{_{DRS3}}}}  \Cll[name]{GBS} \Crr{f1}(j) \Cll[name]{GC}  \end{align*}
where $\Crr{GBS}, \Crr{GC}$ depends only on the good grid.  It follows that the essential spectral radius of $\Phi$ is at most 
$$\big( \liminf_{j\rightarrow \infty} \big( \sum_{P\in   \mathcal{P}^j}   |P|^{1+sp' -\epsilon\gamma_{_{DRS3}}p'}\big)^{1/j}\big)^{1/p'}.$$
and since $\gamma_{_{DRS3}}$ can be taken arbitrarily small we obtained the upper bound 
$$r_{ess}(\Phi)\leq \big( \liminf_{j\rightarrow \infty} \big( \sum_{P\in   \mathcal{P}^j}   |P|^{1+sp'}\big)^{1/j}\big)^{1/p'}.$$
for the essential spectral radius of $\Phi$. It is easy to see that $$r_{ess}(\Phi)\leq (\inf w)^{-s}< 1.$$ 
\end{proof}

\begin{corollary} Suppose that  every branch of $f$ is onto, that is, $f(P)=I$ for every $P\in \mathcal{P}^1$. Consider the transfer operator  defined by 
$$(\tilde{\Phi}\psi)(x)=\sum_{P\in \mathcal{P}^1}   g_P^{1+sp'}  \psi(h_P(x)).$$
Then $\tilde{\Phi}\colon L^\infty \rightarrow L^\infty$ is a bounded operator and
$$r_{ess}(\Phi,\mathcal{B}^s_{p,q})\leq (r(\tilde{\Phi},L^\infty))^{1/p'}.$$
in the case that $f$ is continuous in a topological space with a borelian measure $m$ we have 
$$r(\tilde{\Phi},L^\infty)=e^{P_{top}((1+sp')\log w)},$$
where $P_{top}$ denotes the topological pressure with respect to $f$.
\end{corollary}
\begin{proof} The Bowen condition implies that 
$$\sum_{P\in   \mathcal{P}^j}   |P|^{1+sp'} \leq \Cll{axax} |\tilde{\Phi}^j 1_I|_\infty= \Crr{axax}|\tilde{\Phi}^j|_\infty.$$
for some $\Crr{axax}$ that does not depend on $j$, so we obtain
$$r_{ess}(\Phi,\mathcal{B}^s_{p,q})\leq (r(\tilde{\Phi},L^\infty))^{1/p'}.$$
\end{proof}

A problem with the above approach for Markovian maps is that $\mathcal{B}^s_{p,q}(I,m,\mathcal{P})$ a priori depends on the Markov partition $\mathcal{P}^1$ under consideration (and consequently depends on $f$). So the space $\mathcal{B}^s_{p,q}$  is a {\it ad hoc} space in this approach. On the other hand, we can use this approach with many situations we have a symbolic dynamics acting on subshift of finite type, as in full shifts,  expanding  maps on compact sets (in particular, expanding maps on compact manifolds) without discontinuities and Markov  expanding interval maps.

\section{Conformal expanding  repellers}

Let $I$ be a compact set in  Riemann sphere $\overline{\mathbb{C}}$ endowed with the spherical metric and  suppose that $f\colon I\rightarrow I$ is an (open) expanding repeller  such that $f$ has an conformal extension to a neighbourhood  of $I$ in $\overline{\mathbb{C}}$.  We call $f\colon I \rightarrow I$ a conformal expanding repeller (as in Przytycki and Urba\'{n}ski \cite{pu}). An important example is obtained taking $f$ as a hyperbolic rational map and $I$ as its Julia set.  Let $D$ be the haussdorff dimension of $I$. Let $m$  the $D$-dimensional Haussdorff measure restrict to $I$ and normalized such that $m(I)=1$. Then $|f'|^D$ is the Jacobian of $f$ with respect to $m$, that is, if $A \subset I$ is borellian set with small diameter then
$$m(f(A))=\int_A |f'|^D \ dm.$$
Moreover $m$  is  geometric measure, that is, it satisfies (\ref{geo}) if we take $d$ as the spherical metric. We could now consider a Markov partition for $(f,I)$ and use the methods there using a {\it ad hoc} space $\mathcal{B}^s_{p,q}$. We will use a new method here.  Since $(I,d,m)$ is an Alhfors-regular space (so in particular a homogeneous space), there is a  Besov space $\mathcal{B}^s_{p,q}$ that does not depend one the particular choice of a Markov partition.  Indeed, using Christ \cite{christ} one can construct a good grid $\mathcal{P}$ for $I$ with the following properties. There are constants  $\eta, \Cll{de111a}, \Cll{de1111a},\Cll{de00a},\Cll{ooa}\geq 0$ and $\Cll[c]{de000a}\in (0,1)$ such that for every $Q\in \mathcal{P}^k$, with $k\geq 1$, there is $z_Q\in Q$ satisfying 
\begin{equation}\label{pp} B_{d}(z_Q,\Crr{de111a} \Crr{de000a}^k)  \subset Q,\end{equation}
\begin{equation}\label{pp2} diam_d \  Q \leq \Crr{de1111a} \Crr{de000a}^k  \end{equation}
and
\begin{equation}\label{ppp}  m\{ x\in Q\colon \  d(x,I\setminus Q)\leq \Crr{de00a} t  \Crr{de000a}^k    \}\leq \Crr{ooa} t^{\eta} m(Q).\end{equation} 
See Proposition 2.1 in S. \cite{smania-homo} for details. If $sp < \eta$ then $\mathcal{B}^s_{p,q}(I,m,\mathcal{P})$ coincides with the corresponding Besov space $B^s_{p,q}(I,m,d)$ of the homogeneous space $(I,m,d)$ as defined by  Han, Lu and Yang \cite{han2}. From now one we assume $s \in (0,1)$, $p\geq 1$ and $sp < \min\{1,\eta\}.$ 

\begin{theorem} Let $f\colon I \rightarrow I $ be  a conformal repeller with dimension $D$. Then
$$r_{ess}(\Phi,B^s_{p,q})\leq \min |f'|^{-Ds} < 1.$$
Here $B^s_{p,q}$ is the Besov space  of the homogenous space  $(I,m,d)$.
\end{theorem}
\begin{proof}
Since $f$ is expanding and conformal, there is $\delta_0 > 0$ and $\alpha > 1$  with the following property. For every $x\in I$ we can find a domain $V_x$ such that  
$$f\colon  V_x \rightarrow B_{\overline{\mathbb{C}}}(f(x),2\delta_0)$$
is  conformal and its inverse $h_x$ is a contraction, that  is
$$d_{\overline{\mathbb{C}}}(h_x(z),h_x(y))\leq \frac{1}{\alpha} d_{\overline{\mathbb{C}}}(z,y)$$
for every $x,y\ in B_{\overline{\mathbb{C}}}(f(x),2\delta_0).$ In particular 
$$V_x \subset B_{\overline{\mathbb{C}}}(x,2\delta_0)$$
and we can define the inverse branches of $f^j$
$$h_{j,x}\colon  B_{\overline{\mathbb{C}}}(f^j(x),2\delta_0)\rightarrow  \overline{\mathbb{C}}$$
as
$$h_{j,x}(y)=   h_x\circ \cdots \circ      h_{f^{j-2}(x)} \circ h_{f^{j-1}(x)}(y).$$
Define
$$g_{j,x}(y)= |h_{j,x}'(y)|^D.$$
Using Koebe Lemma  one can prove there is $\Cll{lippp}$ such that for every $x\in I$ and $j$
 $$| g_{x,j}(z) -  g_{x,j}(y)|\leq 2g(y)\Crr{lippp}d(z,y)^{D(\beta+\epsilon)}$$
 for every $y,z\in  B_{\overline{\mathbb{C}}}(f^j(x),\delta_0)$. Using arguments quite similar to those in Section \ref{markov} we can show that  there is $\delta_1 < \delta_0$ such that If  $W, Q\in \mathcal{P}$  satisfy  $Q \subset h_{x,j} B_{\overline{\mathbb{C}}}(f^j(x),\delta_0)$,  $diam \ h_{x,i}^{-1}(Q) < \delta_1$ and   $W\subset h_{x,j}^{-1}Q$ then 
$$|g_{x,j}1_W|_{\mathcal{B}^\beta_{p,q}(W,\mathcal{P}_W,\mathcal{A}^{sz}_{p,q})}  
\leq 2\Cll{uuuu} \frac{|Q|}{|h_{x,j}^{-1}(Q)|}  |W|^{1/p-\beta}.$$
Note that (use Koebe again) there   is $\tilde{k}$ with the following property.  For every $j$ and  $m$-almost every $x \in I$ there is $P_{x,j}\in \cup_k \mathcal{P}^k$ such that $P_{x,j}\subset h_{x,j}(B_{\overline{\mathbb{C}}}(f^j(x),\delta_1/2)$, 
$x\in P_{x,j}$ and $f^j(P_{x,j})$ contains at least one element of $\mathcal{P}^{\tilde{k}}$. In particular there is $\delta_2 > 0$ such that $m(f^j(P_{x,j}))\geq \delta_2$ for every $x,j$. Since $\mathcal{P}$ is a nested sequence of partitions of $I$ one can find a finite family $\Lambda_j \subset \{ P_{x,j}\}_{x\in I}$ that is a partition of $I$. 
For every $R\in \Lambda_j$ denote by $h_R$ and $g_R$ the restrictions of $h_{x,i}$ and $g_{x,i}$ to $R$. 
So if $Q \subset R$ then 
$$|g_R1_W|_{\mathcal{B}^\beta_{p,q}(W,\mathcal{P}_W,\mathcal{A}^{sz}_{p,q})}  
\leq \Crr{DRP}(R) \Big(\frac{|Q|}{|h_R^{-1}(Q)|}\Big)^{1/p-s+\epsilon}   |W|^{1/p-\beta},$$
where
$$ \Crr{DRP}(R) =  \Cll{auxx} \Big(\frac{|R|}{|h_R^{-1}(R)|}\Big)^{1-(1/p-s+\epsilon)},$$
 Note that this map $h_R$  has bounded distortion in the sense that for every $Q\subset P$, with $Q\in \mathcal{P}$ we have that $f^j(Q)$ is an $(1-sp,\Crr{DGD1},\Crr{DGD2})$-regular domain, where $\Crr{DGD1},\Crr{DGD2}$ does not depend on $j$. Indeed  since $f^j$ is conformal, so using Koebe Lemma one can show that in every small open sets close to $I$ the map $f^j$ can be written as $f^j(x)=h( \alpha e^{i\theta} x)$, where $h$ is a bilipchitz function that satisfies (\ref{lip2}), and $\Crr{lm1}$ does not depend on $j$.  Then we can use an argument similar to Proposition \ref{lip}, noticing that the $D$-dimensional Haussdorff measure behaves quite well under the action of scalings and rotations.

We can see the Ruelle-Perron-Frobenious operator of $f^j$ as 
$$(\Phi^j\psi)(x)= \sum_{R\in \Lambda_j}   g_R(x)\psi(h_R(x)).$$
Let $k_{j}=\max\{ k_0(R)\colon \ R\in \Lambda_j\}$. Then  for every $Q\in \mathcal{P}^k$, with $k\geq k_j$ we have that 
\begin{equation}\label{num2} \# \{R\in \Lambda_j \colon R\cap Q\neq \emptyset    \}\leq 1.\end{equation}
 Moreover 
$$\Cll{df}  \Crr{de000a}^{D |k_0(Q)-k_0(h_R^{-1}(Q))|}\leq \frac{|Q|}{|h^{-1}_R(Q)|}  \leq \Crr{DC1}\Crr{de000a}^{ D|k_0(Q)-k_0(h_R^{-1}(Q))|}$$
for  appropriated constants  $\Crr{df},\Crr{DC1}$.  Indeed due the  bounded distortion of $h_R$ we have
$$\Cll{dff}  \Crr{de000a}^{D |k_0(Q)-k_0(h_R^{-1}(Q))|}\leq \frac{|R|}{|h^{-1}_R(R)|}  \leq \Cll{dg}\Crr{de000a}^{D |k_0(Q)-k_0(h_R^{-1}(Q))|}$$
and consequently 
$$ |k_0(Q)-k_0(h_R^{-1}(Q))|\geq  \frac{1}{D\ln \Crr{de000a}} \ln \frac{|R|}{|h^{-1}_R(R)|} -\frac{\ln \Crr{dff}}{D\ln \Crr{de000a}}, $$
so define
$$a_R= \frac{1}{D\ln \Crr{de000a}} \ln \frac{|R|}{|h^{-1}_R(R)|} -\frac{\ln \Crr{dff}}{D\ln \Crr{de000a}}.$$
Take $\Crr{DC2}(R)=\Crr{de000a}^D$. Now we cannot choose   $\Crr{DGD2}$  to be close to zero anymore. But we can choose $\epsilon $ small enough such that 
$$\Crr{DRS2}(R)=\max \{ \Crr{de000a}^{D\epsilon},\Crr{DGD2}^{1/p}\}=\Crr{de000a}^{D\epsilon}  < 1$$
Then the $\Theta_R$ has expression identical to the expression in Section \ref{markov}, so  we can use the same arguments  to conclude that
$$r_{ess}(\Phi)\leq \big( \liminf_{j\rightarrow \infty} \big( \sum_{P\in   \Lambda_j }   |P|^{1+sp'}\big)^{1/j}\big)^{1/p'}.$$
The nature of $\Lambda_j$ is more mysterious here, however  $\Lambda_j$ is a partition of $I$, so we can yet obtain the estimate 
\begin{align*}
&\big( \sum_{P\in   \Lambda_j }   |P|^{1+sp'}\big)^{1/j}\\
&\leq    \big( \sum_{P\in   \Lambda_j }   |P|\big)^{1/j} ( \max_{P\in \Lambda_j}  |P|)^{sp'/j}\\
&\leq  ( \max_{P\in \Lambda_j}  |P|)^{sp'/j}\\
&\leq \delta^{sp'/j}  \big( \max_{P\in \Lambda_j}  \frac{|P|}{|f^j(P)|}\big)^{sp'/j}\\
&\leq \delta^{sp'/j} \min |f'|^{-Dsp'}.
\end{align*} 
so $r_{ess}(\Phi,\mathcal{B}^s_{p,q})\leq \min |f'|^{-Ds} < 1.$
\end{proof}

\section{Piecewise Bi-Lipchitz maps on the interval with  $1/(\beta+\epsilon)$-bounded  \\variation potentials}
\label{bvint}
That is our first class examples of maps that do not have a Markov partition. The study of ergodic theory of piecewise monotone maps on the interval is quite long.   Lasota and Yorke \cite{ly} studied  $C^2$ expanding maps on the interval. Keller and Hofbauer  \cite{hk1}\cite{hk2} considered maps with bounded variation jacobian.  Keller \cite{keller}  studied  of transfer operators with $p$-bounded variation jacobians, which  includes  piecewise $C^{1+}$-diffeomorphisms. In particular he obtained statistical properties for certain bounded observables.

Consider a map
$$f\colon \cup_{i\in \Lambda} I_i \rightarrow I$$
where $I=[0,1]$ and $I_i$  are intervals  and $\Lambda$ is finite, where $\Lambda_1=\{I_i\}$ is a partition of $I$. We assume that $f\colon I_i \rightarrow I$ is continuous and  there is $\alpha,\beta  > 0$ satisfying 
for every $i \in \Lambda$ and $x,y \in I_i$
$$ \beta |x-y| \geq |f(x)-f(y)|\geq \alpha |x-y|.$$
Suppose that $w\colon \cup_i \rightarrow I$ is a $1/(\beta+\epsilon)$-bounded variation  function such that $\inf w > 0$.  Note that $w$ also have finite $1/(\beta+\epsilon')$-bounded variation for every $\epsilon' < \epsilon$, and sometimes it will be useful to reduce $\epsilon$.  Of course 
$f^j$ is also a piecewise expanding map with a corresponding dynamical partition of $I$ by intervals  $\Lambda_j$.

As usual denote by $h_J$ the inverse of $f^j\colon J \rightarrow f^j(J)$ and the induced potential by $g_J$, that has $1/(\beta+\epsilon)$-bounded variation on $J$. Let $\mathcal{D}$ be the dyadic good grid of $I$. Denote

To simplify the notation, denote
$$m(J)=\inf_{\substack{Q\subset J\\ Q\in \mathcal{D}}} \frac{|Q|}{|h_R^{-1}(Q)|}, \ M(J)=\sup_{\substack{Q\subset J\\ Q\in \mathcal{D}}} \frac{|Q|}{|h_R^{-1}(Q)|}.$$
Note that $m(J)$ and $M(J)$ are finite since the branches of $f^j$ are bi-Lipchitzian maps. 

By Proposition \ref{bvp}.A, we can replace  the partition $\Lambda_j$ by a finer finite partition $\Lambda_j$ such that 
for every $W \in \mathcal{D}$ and $W\subset f^j(J)$, with $J\in \Lambda_j$, we have 
\begin{equation}\label{bou} |g_J1_W|_{\mathcal{B}^\beta_{p,q}(W,\mathcal{P}_W,\mathcal{A}^{sz}_{p,q})}  \leq  \Crr{uuug} (\sup_W g_J)   |W|^{1/p-\beta}.\end{equation}
so
$$|g_J1_W|_{\mathcal{B}^\beta_{p,q}(W,\mathcal{P}_W,\mathcal{A}^{sz}_{p,q})}  
\leq \Crr{DRP}(J) \Big(\frac{|Q|}{|h_J^{-1}(Q)|}\Big)^{1/p-s+\epsilon}   |W|^{1/p-\beta},$$
where
$$ \Crr{DRP}(J) = \Crr{uuug}  (\sup_J g_J)  m(J)^{-(1/p-s+\epsilon)},$$
Note that  branches of $f^j$ takes intervals to intervals, and every interval in $I$ is a $(1-sp,\Crr{DGD1},\Crr{DGD2})$-regular domain  and $(1-\beta p,\Cll{srd},0)$-strongly regular domain, where the constants does not depend on the interval.  
For every  interval $A\subset I$  we have 
$$2^{-k_0(A)} \leq  |A|\leq 2^{-(k_0(A)-2)}.$$
Let   $J\in \tilde{\Lambda}_j$ and consider an interval $Q\subset J$. If $j$ is large enough we have  $k_0(h^{-1}_J(Q))< k_0(Q)$  and consequently 
 \begin{align}  \label{jjjj} \frac{1}{4} \Big(  \frac{1}{2} \Big)^{ |k_0(Q)-k_0(h_P^{-1}(Q))|}   &\leq \frac{|Q|}{|h^{-1}_J(Q)|} \nonumber  \\ &\leq   \frac{2^{-(k_0(Q)-2)}}{2^{-k_0(h^{-1}_J(Q))} } \leq 4 \Big(  \frac{1}{2} \Big)^{ |k_0(Q)-k_0(h_P^{-1}(Q))|}.\end{align}
so we can take $ \Crr{DC1}=4$ and $\Crr{DC2}(P)=1/2$. Moreover we can choose
\begin{equation}\label{max} |k_0(Q)-k_0(h_P^{-1}(Q))| \geq a_J = \frac{\ln M(J)}{\ln 2} -2  \geq   j \frac{ \ln \alpha}{\ln 2} -2     .\end{equation} 

We can take $\epsilon$ small enough such that 
$$\Crr{DRS2}(J)=\max \{ (1/2)^{\epsilon},\Crr{DGD2}^{1/p}\}=(1/2)^{\epsilon}  < 1$$
So in this case
\begin{align*}
\Theta_J&=  \Crr{DC1}^{\epsilon} \Crr{DRP}(J)  \Crr{DGD1}^{1/p}  (1/2)^{\epsilon a_J (1-\gamma_{_{DRS3}})} \\
&= \Crr{DC1}^{\epsilon}  \Cll{auxxxx} (\sup_W g_J)  m(J)^{-(1/p-s+\epsilon)} M(J)^{\epsilon(1-\gamma_{_{DRS3}})}\end{align*}
  Moreover there is $k_j$ such that 
 for every $Q\in \mathcal{P}^k$, with $k\geq k_j$ we have that 
\begin{equation}\label{num22} \# \{J\in \Lambda_j \colon J\cap Q\neq \emptyset    \}\leq 2.\end{equation}
By Proposition 11.2 in \cite{smania-transfer}, the Ruelle-Perron-Frobenious $\Phi^j$ of $f^j$  has a $$(\Crr{GSR}\Cll{f22}(j),\Crr{GSR}\Cll{f11}(j),\gamma_{_{DRS3}})$$ essential slicing, for some constant $\Crr{f22}(j)$ and 
$$\Crr{f11}(j)= 2 \Crr{srd}^{1/p}\big( \sum_{J \in \Lambda_j}  \Theta_J^{p'}   \big)^{1/p'} $$
Corollary 10.1 in \cite{smania-transfer} tell us that we can write $\Phi^j= K_j+F_j$, whre $F_j$ is a finite-rank bounded operator in $\mathcal{B}^s_{p,q}$ and 
\begin{align*} |K_j|&\leq \frac{2}{1-(1/2)^{\epsilon \gamma_{_{DRS3}}}}  \Cll[name]{GBS} \Crr{f1}(j) \Cll[name]{GC}. \end{align*}

In particular $\Phi^j$ is a {\it bounded } operator acting on $\mathcal{B}^s_{p,q}$, for every $s< \beta$, $\beta < 1/p$ and $q \in [1,\infty)$. To study the  quasi-compactness of $\Phi^j$  is trickier, since we do not have any control over the quantities $m(J), M(J)$ under the current assumption. So lets assume additionally that $f'$ is also a $1/(\beta+\epsilon)$-bounded variation  function. Then  $h_J'$  is also an $1/(\beta+\epsilon)$-bounded variation and  if necessary we  can refine $\Lambda_j$ such that 
\begin{equation}\label{bou2}  \frac{1}{2}  \leq   \frac{|h_J'(x)|}{|h_J'(y)|}\leq 2\end{equation}
for every $x,y \in h_J(J)$ and $J\in \Lambda_J$. This implies that there is $\Cll{ddd}$ such that
$$  \frac{1}{\Crr{ddd}} \frac{|J|}{|h_J^{-1}(J)|}  \leq     m(J)\leq M(J)\leq \Crr{ddd}\frac{|J|}{|h_J^{-1}(J)|} .$$
and we get the more friendly estimate
$$\Theta_J\leq   \Cll{auxxxxx} (\sup_J g_J)  \Big( \frac{|J|}{|h_J^{-1}(J)|}\Big)^{s-1/p-\epsilon \gamma_{_{DRS3}}}.$$
where  $\Crr{auxxxxx}$ does not depend on $j$. So to get $|K_j|< 1$ on $\mathcal{B}^s_{p,q}$ we need 
$$ 2 \Crr{srd}^{1/p}  \big( \sum_{J \in \Lambda_j}  \Theta_J^{p'}   \big)^{1/p'}  < 1. $$
for the case $p > 1$ and 
$$ 2 \Crr{srd}^{1/p} \sup_{J \in \Lambda_j}  \Theta_J < 1. $$
for the case $p=1$.

\begin{theorem} Consider  $w=|f'|$,  that is, $g_J=|h_J'|$, and $\inf |f'|> 1$.  Then 
$$r_{ess}(\Phi,\mathcal{B}^s_{1,q})\leq (\inf |f'|)^s < 1.$$ We can also have  the quasi-compactness of $\Phi$ on $\mathcal{B}^s_{p,q}$, provided $p$ is close to $1$. 
\end{theorem}
\begin{proof} We have 
$$\sup_J g_J\leq 2 \frac{|J|}{|h_J^{-1}(J)|},$$
and we get  
$$\Theta_J\leq   2 \Crr{auxxxxx}  \Big( \frac{|J|}{|h_J^{-1}(J)|}\Big)^{1-1/p+s-\epsilon \gamma_{_{DRS3}}}.$$
so to get quasicompactness of $\Phi$ on $\mathcal{B}^s_{1,q}$ we need 
$$|K_j|\leq   4 \Crr{srd}  \Crr{auxxxxx} \sup_{J\in \Lambda_j}  \Big( \frac{|J|}{|h_J^{-1}(J)|}\Big)^{s-\epsilon \gamma_{_{DRS3}}}< 1. $$
Note that the constants $ \Crr{srd}, \Crr{auxxxxx}$ may depend on $\epsilon$.  Since  $f$ is expanding ($\inf |f'|> 1$) then
$$|K_j|\leq   4 \Crr{srd}  \Crr{auxxxxx} (\inf |f'|)^{j(s-\epsilon \gamma_{_{DRS3}})}< 1. $$
and we get
$$r_{ess}(\Phi,\mathcal{B}^s_{1,q})\leq \liminf_{j\rightarrow \infty} (4 \Crr{srd} \Crr{auxxxxx} (\inf f')^{j(s-\epsilon \gamma_{_{DRS3}})})^{1/j} = (\inf |f'|)^{s-\epsilon \gamma_{_{DRS3}}}.$$
Taking $\epsilon \rightarrow 0$ we obtain $r_{ess}(\Phi,\mathcal{B}^s_{1,q})\leq (\inf |f'|)^s < 1$. We can also obtain the quasi-compactness of $\Phi$ on $\mathcal{B}^s_{p,q}$, provided $p$ is close to $1$. 
\end{proof}

\section{Continuous $C^{1+\beta+\epsilon}$-piecewise expanding maps on the interval}

 Here we show that 

\begin{theorem} In the setting of the Section \ref{bvint} we may consider the case when $f$ is continuous and every branch of $f$ is a $C^{1+\beta+\epsilon}$-diffeomorfism that is expanding ($\alpha > 1$) and with $w=|f'|$.  Then
$$r_{ess}(\Phi,\mathcal{B}^s_{p,q}) \leq  e^{h_{top}(f)/p'}  (\inf |f'|)^{-(1/p'+s)},$$
where $h_{top}(f)$ denotes the topological entropy of $f$.
In particular $\Phi$ acts as a quasi-compact operator on $\mathcal{B}^s_{p,q}$  for $p\sim 1$.
\end{theorem}
\begin{proof} In Section  \ref{bvint} we started defining $\Lambda_j$ as the partition given by the intervals of monotonicity  of  $f^j$. It is well know that the rate of the grow of the cardinality of $\Lambda_j$ is related with the topological entropy of $f$. Indeed
\begin{equation}\label{jjj} \lim \frac{1}{j} \ln \#\Lambda_j = h_{top}(f).\end{equation} 

But in Section  \ref{bvint} we needed to replace $\Lambda_j$ by a refinement of it. This makes harder to understand the growth of $\#\Lambda_j$. But in the case of continuous $C^{1+\beta+\epsilon}$-piecewise expanding case, we may also need to do a similar procedure, but in a more orderly fashion. Indeed, using the same bounded distortion arguments we used in Section \ref{markov} one can prove that there is $\delta > 0$ such that  for every $j$ and every maximal monotone inverse branch $h_J$ with $J\in \Lambda_j$  we have that  (\ref{bou}) holds  for every $W \in \mathcal{D}$ such that $W\subset f^j(J)$ and $diam \ W < \delta$ and (\ref{bou2}) holds for every $x,y\in f^j(J)$ such that $|x-y|< \delta$. So to obtain  the appropriated  refinement $\tilde{\Lambda}_j$ such that (\ref{bou}) holds $W\subset f^j(J)$   we just need to cut every  interval of $\Lambda_j$ in subintervals in such way that the image of each subinterval by $f^j$ has diameter smaller than $\delta$. This can be made in such way that the number of elements in the new partition has at most the number of elements in the original partition times $1+1/\delta$.   In particular (\ref{jjj}) remains true. Using the results of Section \ref{bvint} one can show that
\begin{align*} r_{ess}(\Phi,\mathcal{B}^s_{p,q}) &\leq \big( \liminf_{j\rightarrow \infty} \big(  \sum_{J\in \Lambda_j}    (\inf_J |(f^j)'|)^{-(1+sp')}\big)^{1/j}  \big)^{1/p'}\\
 &\leq \big( \liminf_{j\rightarrow \infty}    (\#\Lambda_j)^{1/j}  (\inf |f'|)^{-(1+sp')}   \big)^{1/p'}\\
 &\leq e^{h_{top}(f)/p'}  (\inf |f'|)^{-(1/p'+s)}.
 \end{align*}
 \end{proof}
 
\begin{corollary} Consider the tent family $f_t\colon [-1,1]\rightarrow [-1,1]$, with $1/2< t\leq 1$, given by $f_t(x)=-2t|x|+2t-1$. We have $exp(h_{top}(f_t))=2t=|f_t'|.$ So
$$r_{ess}(\Phi_{f_t},\mathcal{B}^s_{p,q}) \leq (2t)^{-s} < 1.$$
for every $s,p,q$ satisfying $0< s< 1/p$, $p\in [1,\infty)$ and $q\in [1,\infty]$. 
\end{corollary}

\section{Piecewise expanding maps on the interval with jacobian in $\mathcal{B}^{1/p}_{p,\infty}$ }

The  class of jacobians for which the quasi-compactness of the transfer operator is obtained in this section is, as far as we know, of the lowest regularity in the literature. The setback is that we need that the dynamics satisfies an a priori estimate. 

\begin{figure}
\includegraphics[scale=0.5]{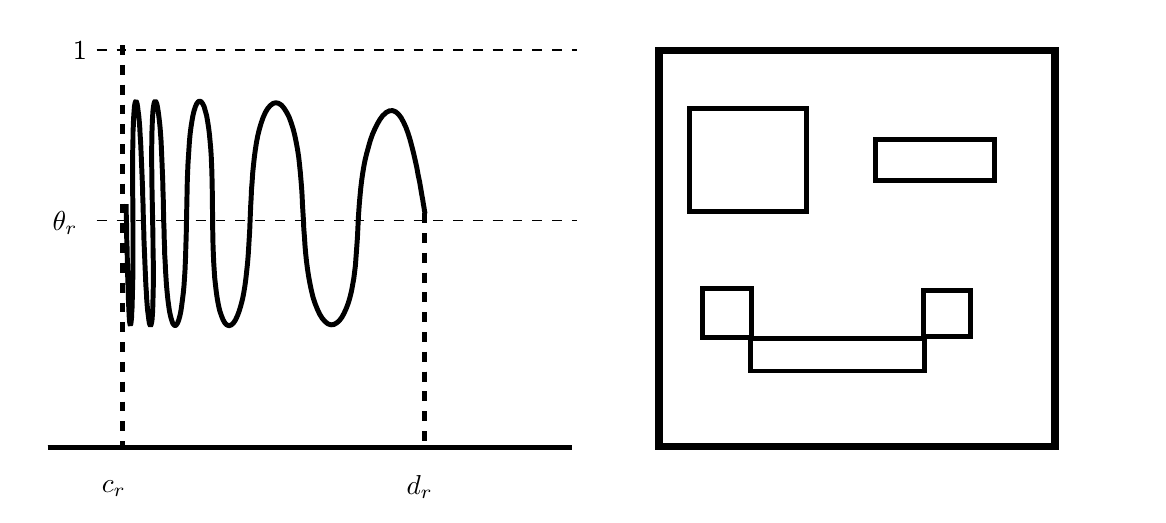}
\caption{{\it On the l.h.s.} Potentials in $\mathcal{B}^{1/p}_{p,\infty}$ can be very irregular. This one, considered in Remark \ref{bosivpt}, has infinite $p$-variation for every $p\geq 1$. {\it On the r.h.s.} The  Winky Face maps  are  a (family of)  examples (see Remark \ref{winky}) of piecewise metric-expanding maps  on the plane for which the transfer operator is quasi-compact.  }
\end{figure}

We consider here $I=[0,1]$ with the good dyadic grid $\mathcal{D}$.

\begin{theorem}\label{hhh} There is $\Cll{apap}> 0$ such that the following holds. Let $\Lambda$ be a finite set, $\theta_r \in (0,1)$ for every $r\in \Lambda$. Let $\{I_r\}_{r\in \Lambda}$ be a finite partition of $I=[0,1]$ by intervals  satisfying $|I_r|\theta_r^{-1} < 1$. 
If   $p > 1$ and $\epsilon_0 > 0$  satisfies 
$$(\#\Lambda)^{1/p'} (\max \{\theta_r\}_r)^{1/p'+s-\epsilon_0} \leq \Crr{apap}$$
 then there is $\delta> 0$ with the following property. \\
 
Choose a collection of intervals $\{J_r\}_{r\in \Lambda}$  in  $I$ such that $\theta_r = |I_r|/|J_r| <  1$ for every $r$. Let $I_r=[a_r,b_r]$, $J_r=[c_r,d_r]$. Suppose that $\alpha_r \colon I \rightarrow \mathbb{R}$ belongs to $\mathcal{B}^{1/p}_{p,\infty}(I)\cap L^\infty$, 
$$\int \alpha_r \ dm=0,$$
$|\alpha_r|_\infty < \min\{1-\theta_r,\theta_r\}$ and $\alpha_r(x)=0$ for every $x\not\in J_r$.  Define
$$h_r\colon J_r \rightarrow I_r$$
as the bi-Lipchitz map
$$h_r(x)=a_r +\int_{[c_r,x]}  \alpha_r + \theta_r \ dm.$$
Consider the piecewise expanding map
$$F\colon \cup_r I_r \rightarrow I$$
defined by $F(x)=h_r^{-1}(x)$ for $x \in I_r$. If 
$$\max_r \{|\alpha_r|_{\mathcal{B}^{1/p}_{p,q}} + |\alpha_r|_{\infty} \} < \delta$$
then 
\begin{itemize}
\item[A.] The transfer operator $\Phi$ is quasi-compact on $\mathcal{B}^s_{p,q}$, 
\item[B.] $\Phi$ satisfies the  Lasota-Yorke inequality for the pair $(\mathcal{B}^s_{p,q},L^1)$,
\item[C.]  $F$ has an absolutely continuous invariant probability $\mu$, 
\item[D.] Every absolutely continuous invariant probability  $\mu= \rho \ m$ satisfies $\rho\in \mathcal{B}^s_{p,q}$.\end{itemize}
\end{theorem} 
\begin{proof} We can assume that $\max \{\theta_r\}_r < 1/10$. The jacobian of $h_r$ is $g_r(x)=  \alpha_r(x) + \theta_r$. 
Recall that  by Proposition \ref{besov-mult} in \cite{smania-besov} there is $\Cll{re}$ such that for every $f\in \mathcal{B}^{1/p}_{p,\infty}(I)\cap L^\infty$ we have 
$$|f1_W|_{\mathcal{B}^\beta_{p,q}(W,\mathcal{D}_W)}\leq (\Crr{re} |f|_{\mathcal{B}^{1/p}_{p,q}} + |f|_\infty)|W|^{1/p-\beta}$$
for every interval $W\in I$.  In particular
$$|g_r1_W|_{\mathcal{B}^\beta_{p,q}(W,\mathcal{D}_W)}\leq (\theta_r+ \Crr{re} |\alpha_r|_{\mathcal{B}^{1/p}_{p,q}} + |\alpha_r|_\infty)|W|^{1/p-\beta}$$
for every $W\subset J_r$, $W\in \mathcal{D}$. 
 If  $\delta$ is small enough then 
\begin{equation}\label{kk}  (1-1/4)\theta_r  \leq \frac{|Q|}{|h_r^{-1}(Q)|} \leq (1+1/4)\theta_r < \frac{1}{8}\end{equation}       
for every interval $Q \subset I_r$, $r\in \Lambda$. Consequently  $|k_0(h_r^{-1}(Q))-k_0(Q)|\geq 1$, so we can take $\Crr{DC1}=1$ and $\Crr{DC2}=1/2$ and $a_r=1$. 
Recall that every interval in  is a $(1-sp,\Crr{DGD1},\Crr{DGD2})$-regular domain, provided we choose $\Crr{DGD1} >0$, $\Crr{DGD2}\in (0,1)$ properly, and in particular this holds for intervals such as $F(Q)$, where $Q$ is an interval inside $I_r$ for some $r$.   Fix $\epsilon \in (0,\epsilon_0)$  such that  $1/p-s+\epsilon < 1$ and
$$\Crr{DRS2}= \max \{ (\Crr{DC2})^{\epsilon},\Crr{DGD2}^{1/p}\}= (1/2)^{\epsilon}.$$
For every $W\subset h_r^{-1}(Q)$ 
we have
\begin{align*} |g_r1_W|_{\mathcal{B}^\beta_{p,q}(W,\mathcal{D}_W)}&\leq \frac{5}{4}\theta_r   |W|^{1/p-\beta} \\
& \leq \frac{5}{4} \Big(\frac{4}{3}\big)^{1/p-s+\epsilon}  \theta_r^{1-(1/p-s+\epsilon)} \Big( \frac{|Q|}{|h_r^{-1}(Q)|}\Big)^{1/p-s+\epsilon}   |W|^{1/p-\beta},\end{align*} 
so we take 
$$\Crr{DRP}(r)=2\theta_r^{1-(1/p-s+\epsilon)}.$$
  So
\begin{align*}\label{thetaest}
\Theta_r&= \Crr{DC1}^{\epsilon}(r) \Crr{DRP}(r)  \Crr{DGD1}^{1/p}  \Crr{DC2}^{\epsilon a_r (1-\gamma_{_{DRS3}})} \nonumber  \\
&\leq 2\Crr{DGD1}^{1/p} \theta_r^{1-(1/p-s+\epsilon)}    (1/2)^{\epsilon (1-\gamma)}.
\end{align*}

Let  $t$ be such that 
$$M=\sup_{\substack{P\in \mathcal{D}^k \\ k\geq t_n}}\# \{r\in \Lambda \colon I_r \cap P\neq \emptyset    \} \leq 2.$$
Of course 
\begin{equation}\label{nv} N=\sup_{P\in \mathcal{H}} \#\{ r\in \Lambda \ s.t. \  P\subset J_r \}\leq \#\Lambda\end{equation}
Note that every interval, and in particular the intervals in the partition  $\{ I_r\}_r$ are 
$(1-\beta p,\Crr{srd2},t)$-strongly regular  domain, for some universal $\Crr{srd2}$. 

We can apply Theorem \ref{trans-pa2} in  \cite{smania-transfer} with $p=1$  to conclude that $\Phi$ has a  $$(\Crr{GSR}\Cll{f2at},\Crr{GSR}\Crr{f1at},t)$$ essential slicing on $\mathcal{B}^s_{1.q}$,  for some $\Crr{f2at}$ and 
\begin{align*}\Cll{f1at}&=  \Crr{srd2}^{1/p}2\Crr{DGD1}^{1/p}(\#\Lambda)^{1/p'} \theta_r^{1-(1/p-s+\epsilon)}    (1/2)^{\epsilon (1-\gamma)},
\end{align*}
and Corollary \ref{trans-ww}  therein tell us that  $\Phi$ has a similar upper bound for its essential spectral radius bounded and consequently  if $$(\#\Lambda)^{1/p'} (\max \{\theta_r\}_r)^{1/p'+s-\epsilon_0}$$ is small enough we can apply Theorem \ref{trans-t1} and  Corollaries \ref{trans-spec} and \ref{trans-phy} in \cite{smania-transfer} to conclude that 
$$\sigma_{ess}(\Phi, \mathcal{B}^s_{p,q}) < 1,$$ and we  can obtain the Lasota-Yorke inequality applying Theorem \ref{trans-t1} and  Corollaries \ref{trans-spec} and \ref{trans-phy}, and Theorem \ref{trans-acim} in \cite{smania-transfer}. 
\end{proof}

\begin{remark}\label{bosivpt} Let $p > 1$. Potentials in $\mathcal{B}^{1/p}_{p,\infty}$ can be very irregular. Consider for instance
$$\alpha\colon [0,1/2]\rightarrow [-1,1]$$ given by
$\alpha(x)=\sin(2\pi \frac{\log x}{\log 2})$. Then $\alpha \in \mathcal{B}^{1/p}_{p,\infty}(I)$. Indeed if $\alpha_+(x)=\max\{\alpha(x),0\}$
$$\alpha_+=\sum_{k} \sum_{Q\in \mathcal{D}^k} c_Q a_Q$$
where $a_Q(x)=1_P(x)$ is the $(1/p,p)$-Souza's atom supported on  $Q$, $c_I=0$ and
$$c_Q= \inf_Q \alpha_+   -\inf_{Q'}  \alpha_+,$$
if $Q\in \mathcal{D}^k$ with $k\geq 1$ and $Q'$ is such that $Q\subset Q'$ and  $Q'\in\mathcal{D}^{k-1}$. If $Q\subset [2^{-(i+1)},2^{-i}]$ we have  
$$|c_Q|\leq C2^{i}|Q|,$$
and $c_Q=0$ for every $Q$ such that $Q\not\subset [2^{-(i+1)},2^{-i}]$ for every $i$, 
so
\begin{align*}\sum_{Q\in \mathcal{D}^k} |c_Q|^p&= \sum_{i < k}\sum_{\substack{Q\in \mathcal{D}^k\\ Q\subset  [2^{-(i+1)},2^{-i}]}} |c_Q|^p\\
&\leq C \sum_{i < k}  2^{(i-k)p+(k-i)}\\
&\leq  C\sum_{i < k}  2^{(k-i)(1-p)}\leq \frac{C}{1-2^{1-p}}.
\end{align*}
and consequently $\alpha_+ \in \mathcal{B}^{1/p}_{p,\infty}$. In the analogous way $\alpha_-(x)=\max\{-\alpha(x),0\}\in \mathcal{B}^{1/p}_{p,\infty}$. So $\alpha=\alpha_+-\alpha_-\in \mathcal{B}^{1/p}_{p,\infty}$. In particular we can construct $\alpha_r$ using $\alpha$ in such way to satisfy Theorem \ref{hhh} (See Figure 1). 
\end{remark}

\section{Infinitely many branches with  small images} 

There are three main motivations for the family of expanding maps we study in this section.  First,  we provided examples of maps with  jacobian in $\mathcal{B}^{1/p}_{p,\infty}$ but without absolutely continuous invariant probabilities and whose transfer operator is not quasi-compact on $\mathcal{B}^s_{p,q}$. Secondly it also provides examples of expanding maps  with infinitely many branches and whose images of most of the branches are very small, and yet the transfer operator is quasi-compact. Moreover the maps in this family  have a  dynamical behaviour quite similar to induced maps of unimodal maps that appears  in the study of the existence of wild attractors. See Bruin,  Keller,  Nowicki  and van Strien \cite{wild1}, Keller and Nowicki \cite{wild3}, Bruin, Keller and  St. Pierre \cite{wild2},  and Moreira and S. \cite{wild4}. 

\begin{figure}
\includegraphics[scale=0.32]{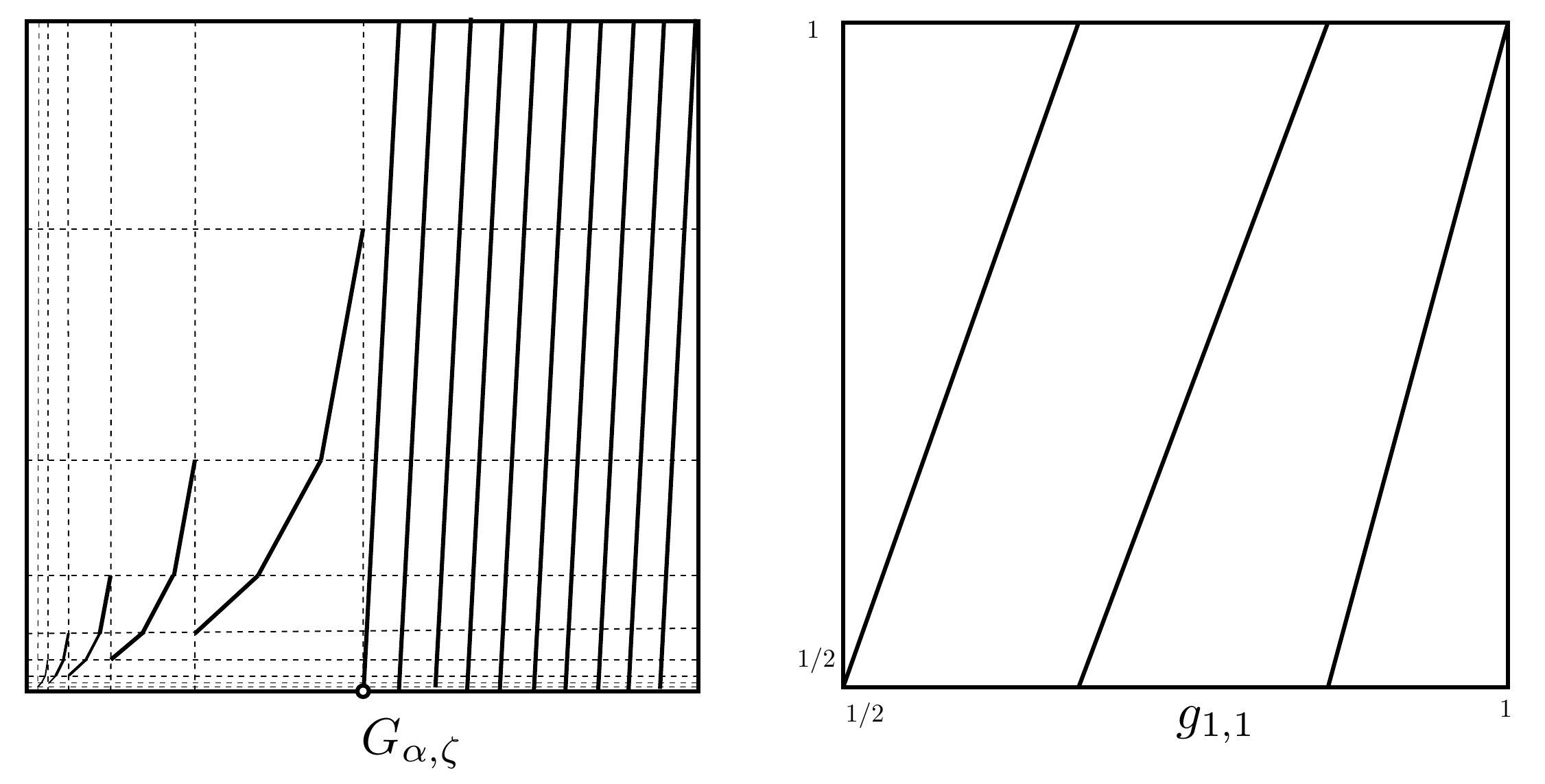}
\caption{{\it On the l.h.s.}  Example of  a map $F_{\alpha,\zeta,k_0}$. {\it On the r.h.s.} Close to $0$ the map $F_{1,1,k_0}$ is conjugate to a skew-product generated by the map $g_{1,1}$. }
\end{figure}
For every $r\geq 0$ define $L_r =[2^{-(r+1)},2^{-r}]$. Note that $L_r \in \mathcal{D}^{r+1}$. 

\begin{lemma}\label{wildd}  There is $\Cll{sai}$ such that for every $\Lambda \subset \mathbb{N}$ the set 
$$\Omega=\cup_{i\in \Lambda} L_r$$
is a $(1- \beta p, \Crr{sai},0)$-strongly regular domain.
\end{lemma} 
\begin{proof} Given $Q\in \mathcal{D}^k$, for some $k$, there are three cases. \\ \\ 
\noindent {\it Case i.} Suppose $Q\neq [0,2^{-k}]$ and $Q\subset L_r$ for some $r\notin \Lambda$. In this case  $Q\cap \Omega=\emptyset$, so pick $\mathcal{F}^i(\Omega\cap Q)=\emptyset$ for every  $i$, \\ \\
\noindent {\it Case ii.}  Suppose $Q\neq [0,2^{-k}]$ and $Q\subset L_r$ for some $r\in \Lambda$, so $Q\cap \Omega=Q$. In this case  pick $\mathcal{F}^k(\Omega\cap Q)=\{Q\}$ and $\mathcal{F}^i(\Omega\cap Q)=\emptyset$ for every  $i\neq k$.\\ \\
\noindent {\it Case iii.} If $Q=[0,2^{-k}]$ then pick either 
$$\mathcal{F}^i(\Omega\cap Q)=\{ L_{i+1}\}  $$
if $i+1 \in  \Lambda$ and $i+1\geq k$, or $\mathcal{F}^i(\Omega\cap Q)=\emptyset$ otherwise. \\
In all cases 
$$Q\cap \Omega =\cup_i \cup_{P\in \mathcal{F}^i(Q\cup \Omega)} P$$
and
$$\sum_{P\in \mathcal{F}^i(\Omega\cap Q)} |P|^{1-\beta p}\leq \sum_{i \geq k} 2^{-i(1-\beta p)}\leq \Crr{sai} 2^{-k(1-\beta p)}=\Crr{sai}  |Q|^{1-\beta p}.$$
This proves the Lemma. 
\end{proof} 

Let $i_0\in \mathbb{N}$. For each  $j=0,1,2,3$ let

$$I_{j}^{i_0}=\bigcup_{\substack{_{i\geq i_0} \\  _{i=j \ mod  \ 3}}} L_i.$$
For $\alpha> 0$  define
$$G_{\alpha,\zeta,i}\colon L_i \rightarrow I$$
as \begin{align*}
&G_{\alpha,\zeta,i}(x)\\
&=\begin{cases}
{\tiny\alpha ( 7/2-13\zeta /6) (x - 1/2^{i+1}) + \alpha/2^{i+2}},  &for \ 1/2 \leq 2^{i} x<11/16, \\
{\tiny \alpha ( 7/2-5\zeta /6) (x - 11/2^{i+4}) + \alpha( 1/2^{i+1} -13\zeta/2^{i+5})},  &for \ 11/16 \leq 2^{i} x< 7/8,   \\
{\tiny \alpha( 7/2+9\zeta /2) (x - 1/2^{i}) + \alpha/2^{i-1},} &for \  7/8 \leq 2^{i} x < 1.
\end{cases} 
\end{align*}
Note that $G_{\alpha,\zeta,i}$ is continuous, monotone  and $|G_{\alpha,\zeta,i}'|\geq 4\alpha/3$
for every $\alpha>0$ and $\zeta\in [0,1]$. Moreover $G_{1,\zeta,i}L_i = L_{i+1}\cup L_i\cup L_{i-1}.$
Define
$$F_{\alpha,\zeta, j,i_0}\colon I_{j}^{i_0} \rightarrow  I$$
as $F_{\alpha,\zeta, j,i_0}(x)=G_{\alpha,\zeta, i}(x)$ for $x\in L_i\subset I_{j}^{i_0}.$ For every $\alpha > 0$ we have that $F_{\alpha,\zeta, j,i_0}$ is injective. 

Fix $\alpha > 0$ and choose 
$$i_0=\min \{ i\geq 0\colon \alpha  2^{-i+1}\leq 1  \}.$$
Divide the interval $[2^{-i_0},1]$ in $k_0$  intervals $I_{-1},\dots, I_{-k_0}$  of same size, 
$$\Lambda=\{-k_0, \dots, -1, 0,1,2,3\}$$
and $I_{j}= I_{j}^{i_0}$ for $j\geq 0$.  Consider the piecewise linear map
$$F_{\alpha,\zeta,k_0}\colon \cup_{r\in \Lambda} I_r \rightarrow I$$
where $F_{\alpha,\zeta,k_0}(x)=G_{\alpha,\zeta, j,i_0}(x)$ for $x\in I_j$, if $j\geq 0$, and
$$F_{\alpha,\zeta,k_0}\colon I_j \rightarrow I$$
is an onto affine map if $j < 0$. In particular $|G_{\alpha,\zeta,k_0}'|\geq  k_0$ on $I_j$ for $j<0$. Denote the inverse of $F_{\alpha,\zeta,k_0}$ on  $I_j$ by $h_j$ and $J_j= F_{\alpha,\zeta,k_0}(I_j)$.

\begin{theorem} 
We have
\begin{itemize}
\item[A.] The Dirac mass supported in $\{0\}$ is the unique physical measure of $G_{1,1,k_0}$ and its basin of attraction is the whole $I$ (up to a set of zero Lebesgue measure), so in particular for every $n$
$$\Phi^n\colon \mathcal{B}^s_{p,q} \rightarrow \mathcal{B}^s_{p,q}$$ is not a quasi-compact operator and it does not satisfy  the Lasota-Yorke inequality for the pair $(\mathcal{B}^s_{p,q},L^1)$.
\item[B.] For every $\alpha> 0$, $\zeta \in (0,1]$, $p\in [1,\infty)$, $q\in [1,\infty)$   we have that  $\Phi$ is a bounded operator acting on $\mathcal{B}^s_{p,q}$. 
\item[C.]  Let $p\in (1,\infty)$  and $q\in [1,\infty]$. If $k_0$  and $\alpha$  are  large enough and $\zeta \in [0,1]$   then $\Phi$ is a quasi-compact operator on $\mathcal{B}^s_{p,q}$ and it satisfies the Lasota-Yorke inequality for the pair $(\mathcal{B}^s_{p,q},L^1)$.
\end{itemize} 
\end{theorem} 
\begin{proof}[Proof of A] Note that $G_{1,1,k_0}$ is an expanding  markovian map. It is not difficult to show that $G_{1,1,k_0}$ has the following property: if $S$ satisfies  $G_{1,1,k_0}^{-1}(S)=S$ then either $m(S)=0$, $m(S)=1$ or 
$$S\subset \{x\in I\colon \  \lim_n G_{1,1,k_0}^n(x)=0\}.$$
So if  the basin of attraction of the Dirac mass supported in $\{0\}$ has positive Lebesgue measure then it has full Lebesgue measure since its complement is backward invariant. 
Note that $G_{1,1,k_0}$ is conjugated with a skew product close to $0$. Indeed 

Define
$g_{1,1}\colon [1/2,1]\rightarrow [1/2,1]$
as
$$ 
g_{1,1}(x)=\begin{cases}
8x/3 - 5/6 , &   1/2 \leq x<11/16, \\
8x/3  - 4/3&  11/16 \leq x< 7/8,   \\
4x   -3 &  7/8 \leq x< 1.
\end{cases} 
$$
and
$\psi \colon [1/2,1]\rightarrow \{ -1,0,1\}$
as
$$ 
\psi(x)=\begin{cases}
-1 , &   1/2 \leq x<11/16, \\
0&  11/16 \leq x< 7/8,   \\
1&  7/8 \leq x< 1.
\end{cases} 
$$
Consider the  skew product 
$$T\colon [1/2,1]\times \mathbb{Z}\rightarrow [1/2,1]\times \mathbb{Z}$$
given by $T(x,i)=(g_{1,1}(x), i+\psi(x))$. Define
$$H\colon  [1/2,1]\times \mathbb{Z} \rightarrow (0,+\infty)$$
by $H(x,i)=2^{-i}x$. Then    $H\circ T(x,i)= G_{1,1,k_0}\circ H(x,i)$, provided $i$ is large enough. Note that the (normalised)  Lebesgue measure $m$ on $[1/2,1]$ is the unique absolutely continuous invariant probability  of $g_{1,1}$ and $m$ is ergodic. Since
$$\int \psi \ dm > 0$$
it follows that for $x \in [0,1]$ in a subset $S$ of positive Lebesgue measure we have
$$\lim_{n\rightarrow \infty} G_{1,1,k_0}^n(x)=0,$$
which implies that $S$ has full Lebesgue measure.  So the Dirac mass at $0$ is the unique physical measure of $G_{1,1,k_0}$ and its basin of attraction has full Lebesgue measure.
It follows that $\Phi^n$ is not quasi-compact on $\mathcal{B}^s_{p,q}$, otherwise $G_{1,1,k_0}$ would have an absolutely continuous invariant probability. 
\end{proof} 

\begin{proof}[Proof of B] Choose $k_0$ large enough such  for every $\alpha > 0$ and $\zeta \in [0,1]$ we have  that the transfer operator $\Phi_1$  of the restriction
$$F_{\alpha,\zeta,k_0}\colon \cup_{i< 0  }I_i \rightarrow I$$
can be written has
$$\Phi_1= T_1+ K_1,$$
where $T_1$ is a finite-rank operator, $|K_1|_{\mathcal{B}^s_{p,q}}< 1/4$and it also satisfies the  Lasota-Yorke inequality
$$|\Phi_1 f|_{\mathcal{B}^s_{p,q}}\leq 4^{-1} |f|_{\mathcal{B}^s_{p,q}} + \Cll{yyy}|f|_1.$$
Consider the  transfer operator $\Phi_2$  of the restriction
$$F_{\alpha,\zeta,k_0}\colon \cup_{i\geq 0  }I_i \rightarrow I$$
By Lemma \ref{wildd} every domain $I_{j}$, with $j=0,1,2,3$, is a $(1- \beta p, \Crr{sai},0)$-strongly regular domain. For every $P=[a,b]\in \mathcal{D}^k$ such that $P\subset J_i$ there is $[c,d]\subset  [a,b]$   such that
$$h_i'\cdot 1_P =  \frac{1}{\alpha} \big( (7/2-13\zeta /6)^{-1}1_{[a,c]} + ( 7/2-5\zeta /6 )^{-1}1_{[c,d]} + (7/2+9\zeta /2)^{-1}1_{[d,b]} \big)
$$
so if $\zeta \in [0,1]$ then
$$|h_i'\cdot 1_W|_{\mathcal{B}^\beta_{p,q}}\leq  \frac{\Cll{tt}}{\alpha} |W|^{1/p-\beta} \leq  \frac{\Cll{poi}}{\alpha^{1-1/p +s-\epsilon}}\big(\frac{|W|}{|h_i^{-1}(W)|}\Big)^{1/p-s+\epsilon} |W|^{1/p-\beta}  .$$
Let $\Crr{DRP}=\Crr{poi}/\alpha^{1-1/p +s-\epsilon}$. We can take as usual $\Crr{DC1}=1$ and $\Crr{DC2}=1/2$ and 
$$a_r=-\frac{\log \alpha}{\log 2}- \Cll{yu},$$
Fix $\gamma=1/2$. Choosing $\epsilon $ small enough we have
$$
\Theta_i= \Cll{wwww} 1/\alpha^{1-1/p +s-\epsilon/2} $$
for every $i=0,1,2,3.$
By Theorem \ref{trans-tt1} in \cite{smania-transfer} we have
that $\Phi_2$ is a bounded operator in $\mathcal{B}^s_{p,q}$  with
$$|\Phi_2|_{\mathcal{B}^s_{p,q}}\leq \Cll{nnn}/\alpha^{1-1/p +s-\epsilon/2},$$
so if $\alpha$ is large enough then $|\Phi_2|_{\mathcal{B}^s_{p,q}}\leq 1/4$ and consequently $\Phi= \Phi_1+\Phi_2$ is a quasi-compact operator and it satisfies the Lasota-Yorke inequality for the pair $(\mathcal{B}^s_{p,q},L^1)$. 
\end{proof}

\section{Lorenz maps with non-flat singularities} 

One of the the motivation to the results Keller \cite{keller} is to study Lorenz maps, an important class of examples that appears in the study of singular hyperbolic flows. Here we  obtain the quasi-compactness in a spaces of functions with more general class of observables, that includes unbounded ones.

\begin{proposition} \label{purefolding}
Let $\Lambda_1$ be a collection of pairwise disjoint intervals of $\hat{I}=[a,b]$ and
$$F\colon \cup_{J\in \Lambda_1} J \rightarrow \hat{I}$$
be a map with following property.  For every $J \in \Lambda_1$ we have that the restriction of $F$ to $J$  is
\begin{itemize}
\item  either   a $C^{1+\beta+\epsilon}$- diffeomorphism,
\item  or  
$$F(x)=     (D_Jx-a_J)^{1/(1+\gamma)}.$$
for $x\in J=[a_J,b_J]$,  $D_J > 0$ and $\beta < \min\{1,\gamma\}.$ In this case we say that $F\colon J \rightarrow F(J)$ is a {\it Lorenz branch.}
\end{itemize} 
Then the Ruelle-Perron-Frobenious  operator $\Phi$ of $F$ with $g_J=|h_J'|$  is a bounded operator in  $\mathcal{B}^s_{1,q}=\mathcal{B}^s_{1,q}(\hat{I},\mathcal{D}, m)$ that can be written as
$$\Phi=Z_F+K_F$$
where $Z_F$ is a bounded finite rank operator on $\mathcal{B}^s_{1,q}$ and

\begin{equation}\label{w23} |K_F|_{\mathcal{B}^s_{1,q}}\leq  \frac{2}{1-(1/2)^{\epsilon \gamma_{_{DRS3}}}}  \Crr{GBS}4^{\epsilon+1} (2\Crr{uuu}+ \Crr{novac}) \alpha^{-(s+\epsilon\gamma_{_{DRS3}})} \Crr{GC}.\end{equation}
Here $\alpha = \inf |F'|.$ In partcular if  $\alpha$ is large enough we have that $\Phi$ is quasi-compact in $\mathcal{B}^s_{1,q}$.
\end{proposition} 
\begin{proof} To deal with the non-Lorenz branches, we will use methods similar  to those in Section \ref{bvint}. Proposition \ref{holderr}.B we  can refine $\Lambda_1$ (that is, replace intervals in $\Lambda_1$  by finite collections of pairwise disjoint intervals that covers the original intervals) in such way that for every non-Lorenz branch $F\colon J \rightarrow f(J)$ we have 
\begin{align*} |g_J1_W|_{\mathcal{B}^\beta_{p,q}(W,\mathcal{P}_W,\mathcal{A}^{sz}_{p,q})}  
&\leq 2\Crr{uuu} \frac{|Q|}{|h^{-1}(Q)|}  |W|^{1/p-\beta}\\
&\leq 2\Crr{uuu}\alpha^{-(1-1/p+s-\epsilon)}  \Big(\frac{|Q|}{|h^{-1}(Q)|}\Big)^{1/p-s +\epsilon}    |W|^{1/p-\beta}.  \end{align*}
for every $Q,J\in \mathcal{D}$ such that $Q \subset J$ and $W\subset f(Q)$. Here $g_J=|h_J'|$ and  $h_J$ is the corresponding inverse branch and $\Crr{uuu}$ depends only on the good grid $\mathcal{D}$. 

On the other hand if $J\in \Lambda_1$ is a Lorenz branch then  by Proposition \ref{lorenzz} we have
\begin{align*} |g1_W|_{\mathcal{B}^\beta_{p,q}(W,\mathcal{P}_W,\mathcal{A}^{sz}_{p,q})}  &\leq    \Crr{novac} \frac{|Q|}{|h^{-1}Q|} |W|^{1/p-\beta} \\
&\leq  \Crr{novac} \alpha^{-(1-1/p+s-\epsilon)}  \Big(\frac{|Q|}{|h^{-1}(Q)|}\Big)^{1/p-s +\epsilon}    |W|^{1/p-\beta}.
\end{align*} 
where $Q$ and $W$ satisfy the same conditions as before. 

Note that (\ref{jjjj}) and (\ref{max}) also holds (with $j=1$) for every $J \in \Lambda_1$.  Taking $\epsilon$ small enough we obtain
\begin{align*}
\Theta_J&=  \Crr{DC1}^{\epsilon} \Crr{DRP}(J)  \Crr{DGD1}^{1/p}  (1/2)^{\epsilon a_J (1-\gamma_{_{DRS3}})} \\
&\leq  4^{\epsilon+1} (2\Crr{uuu}+ \Crr{novac}) \alpha^{-(1-1/p+s-\epsilon)}  \alpha ^{\epsilon(1-\gamma_{_{DRS3}})}\end{align*}
and  consequently (\ref{w23}) holds.
\end{proof}

\begin{corollary}  Let 
$$f\colon \cup_{J\in \Lambda_1} J \rightarrow I$$
be a map with following property. For every $J \in \Lambda_1$ we have that the restriction of $f$ to $J$  is
\begin{itemize}
\item[i.](Branch Type I) either   a $C^{1+\beta+\epsilon}$- diffeomorphism,
\item[ii.](Branch Type II)  or  
$$f(x)=   \psi_J ((D_J\phi_J(x)-a_J)^{1/(1+\gamma)}).$$
for $x\in J=[a_J,b_J]$ and $\beta < \min\{1,\gamma\}$,  and  $\phi_J$, $\psi_J$ are $C^{1+\beta+\epsilon}$- diffeomorphisms. \end{itemize} 
If $\alpha = \inf |f'|  > 1$ then $\Phi$ is a  quasi-compact operator acting on  $\mathcal{B}^s_{1,q}$ and satisfying 
\begin{equation}\label{llore} r_{ess}(\Phi,\mathcal{B}^s_{1,q})\leq \alpha^{-s}.\end{equation} 
\end{corollary} 
\begin{proof} 

 Let $I=[c,d]$. We claim  that there are  two functions $F_1\colon A \rightarrow I$ and $F_2\colon B \rightarrow \hat{I}$, where $\hat{I}=[c,2d-c]$ such that 
\begin{itemize}
\item[-]  We have that $I \subset A\cup B \subset   \hat{I}=[a,\hat{b}]$. Moreover  $A$ and $B$ are disjoint and $|\hat{I}|=2^t|I|$ for some $t\in \mathbb{N}$, 
\item[-] Both $F_1$ and $F_2$ satisfy the assumptions of Proposition  \ref{purefolding},  $\inf |F'_1| > \alpha^j$ and $\inf |F'_1| > \alpha.$
\item[-]  We have
$$f^j(x)= \begin{cases}   F_1(x),   & $if$ \  x\in A,\\
F_2^{2j+1}(x), &if \ x\in B\cap I.
\end{cases}    
 $$
\end{itemize} 

In particular if   $$\Phi_{F_i}\colon \mathcal{B}^s_{1,q}(\hat{I},\mathcal{D}_{\hat{I}},m) \rightarrow \mathcal{B}^s_{1,q}(\hat{I},\mathcal{D}_{\hat{I}},m)$$ is the Ruelle-Perron-Frobenious of $F_i$, with $i=1,2$ then for every small $\epsilon, \gamma_{_{DRS3}}$   we can write
$$ \Phi_{F_1} + \Phi_{F_2}^j =  Z_j + K_j,$$
where $Z_j$ has finite rank and 
$$|K_j|_{\mathcal{B}^s_{1,q}(\hat{I},\mathcal{D}_{\hat{I}},m)}\leq  C \alpha^{-j(s+\epsilon\gamma_{_{DRS3}})},$$
where the constant $C$ may depends on $\epsilon$ but does not depend on $j$.  It is easy to see that the inclusion
$$\mathcal{I}\colon \mathcal{B}^s_{1,q}(I,\mathcal{D}_{I},m) \rightarrow  \mathcal{B}^s_{1,q}(\hat{I},\mathcal{D}_{\hat{I}},m)$$
given by $\mathcal{I}(\psi)=\psi$ is continuous, as well the  multiplier
$$\mathcal{M}\colon \mathcal{B}^s_{1,q}(\hat{I},\mathcal{D}_{\hat{I}},m) \rightarrow \mathcal{B}^s_{1,q}(I,\mathcal{D}_{I},m)$$
given by $\mathcal{M}(\psi)=1_I \psi$.

If $$\Phi_{f}\colon \mathcal{B}^s_{1,q}(I,\mathcal{D}_{I},m) \rightarrow \mathcal{B}^s_{1,q}(I,\mathcal{D}_{I},m)$$ is the Ruelle-Perron-Frobenious of $f$  we have
$$\Phi^j_f  \psi = \Phi_{F_1}\psi + \Phi_{F_2}^j\psi$$
for every $\psi\in L^1(m)$ with support on $I$.  So
$$\Phi^j_f =   \mathcal{M} Z_j\mathcal{I}  + \mathcal{M} K_j \mathcal{I},$$
Since  $\mathcal{M} Z_j\mathcal{I}$ is a finite rank operator and $ \mathcal{M} K_j \mathcal{I}$ and
$$ |\mathcal{M} K_j \mathcal{I}|_{\mathcal{B}^s_{1,q}(I,\mathcal{D}_{I},m)}\leq |\mathcal{M}| C \alpha^{-j(s+\epsilon\gamma_{_{DRS3}})} |\mathcal{I}|$$
that implies 
$$r_{ess}(\Phi_f, \mathcal{B}^s_{1,q}(I,\mathcal{D}_{I},m))\leq \alpha^{-s}.$$

 It remains to prove the claim. Let $J_1,\dots J_k$ be an enumeration of the elements $J\in \Lambda$ such that  $f^j\colon J \rightarrow f^j(J)$ is a branch of  Type II.  We are going to define recursively a family of intervals $\{U_1,\dots,U_k\}$,  intervals
$$I_0\subset I_1 \subset \cdots \subset I_k$$
and respective partitions $\Lambda^i$ of $I_i$ by intervals and maps
$$f_i\colon \cup_{J\in \Lambda^i} \rightarrow I_i$$
such that  $J_\ell \in \Lambda_i$  and $f_i=f^j$ on $J_\ell$ for every $\ell > i$.  
Define $f_0=f^j$, $I_0=I$ and  $\Lambda^0=\Lambda_j$. Assume we have defined $f_i,I_i,$ and $\Lambda^i$. Let $J_{i+1}=[a_1,b_1]$. We have that  $f^j=f_i \colon J_{i+1}\rightarrow f^j(J_{i+1})$   can be written as 
$$f^j= \psi_n\circ E_n \circ \phi_n \circ \psi_{n-1} \circ E_{n-1} \circ \phi_{n-1} \circ  \cdots  \psi_{1} \circ E_1 \circ \phi_1,$$
where $1\leq n\leq j$, Here    
$$\phi_i\colon [a_\ell,b_\ell]\rightarrow [0,c_\ell]$$ and
$$\phi_\ell\colon [0,E_\ell(c_\ell)]\rightarrow [a_{\ell+1},b_{\ell+1}]$$
are  $C^{1+\beta+\epsilon}$- diffeomorphisms, the  map $E_\ell$ is  defined by 
$E_\ell(x)= (D_\ell x)^{1/(1+\gamma_\ell)}$, with $D_\ell \in \{D_J\}_{J\in \Lambda_1}$, and $a_\ell=0$ for $1<    \ell < n$, and moreover there is $\Cll{ddda}$ such that 
$$\frac{1}{\Crr{ddda}}\leq \phi_\ell(x), \psi_\ell(x) \leq \Crr{ddda},$$ 
for every $j$, and every branch $J\in \Lambda_j$ of $f^j$  that is not a diffeomorphism, and every $\ell< n$.   

Choose $A$ such that 
$A/\Crr{ddda} = \alpha^{2j}$.  Define $\omega_1(x)= A\phi_1(x)$,  $\omega_2(x)= A \phi_2 \circ \psi_1(Ax)$, $\dots$, $\omega_\ell(x)= A \phi_\ell \circ \psi_{\ell-1}(Ax)$ and $\omega_{n+1}(x)=\psi_n(Ax).$ Also  set $\tilde{E}_\ell(x) =  (D_\ell A^{-2-\gamma_\ell}x)^{1/(1+\gamma_\ell)}.$ Of course 
$$f^j(x)=  \omega_{n+1}\circ   \tilde{E}_n\circ   \omega_n     \circ \cdots                       \circ \tilde{E}_2\circ   \omega_2  \circ  \tilde{E}_1\circ  \omega_1(x).$$
Assume that $\phi_0(a_1)=0$ (the case $\phi_0(b_1)=0$ is similar).  Now define $U_i=[a_1+\delta,b_1]$, $\tilde{R}_ 0=[a_1,a_1+\delta]$,    $R_1=\omega_1([a_1,a_1+\delta])$, and recursively  $\tilde{R}_\ell= \tilde{E}_\ell(R_\ell)$ and $R_{\ell+1}= \omega_{\ell+1}(\tilde{R}_\ell)$. If $\delta > 0$ is small enough  we have 
$$\min_\ell \min\{ \min_{x\in \tilde{R}_{\ell-1}} \omega_\ell'(x),  \min_{x\in R_{\ell}} \tilde{E}_\ell'(x)\} > \alpha^{2j}$$
and
\begin{equation} \label{small}2j\alpha^{2j}|\tilde{R}_0| +\sum_\ell |\tilde{R}_\ell|+\sum_\ell |R_\ell| <   2^{-i}|I|\end{equation}

Let $\pi_\ell$, with $1\leq \ell\leq n+1$, and  $\tilde{\pi}_\ell$, with $0\leq \ell\leq n$,  be affine isometries,   where  $\pi_{n+1}(x)=\tilde{\pi}_0(x)=x$, and such that 
 $$\{ \pi_\ell(R_\ell)\}_{1\leq \ell\leq n}\cup\{ \tilde{\pi}_\ell(\tilde{R}_\ell)\}_{1\leq \ell \leq n}$$ is a family of pairwise disjoint intervals outside $I_i$ such that 
$$\tilde{\Lambda}^i=( \Lambda^i\setminus \{ J\}) \cup \{ [a_1+\delta,b_1], \tilde{R}_ 0 \} \cup \{ \pi_\ell(R_\ell)\}_{1\leq \ell\leq n}\cup\{ \tilde{\pi}_\ell(\tilde{R}_\ell)\}_{1\leq \ell \leq n}$$
is a partition of an interval $\tilde{I}_i=[c,d_i]\supset I_i$.  Note that  $f^j$ is a diffeomorphism on the   interval  $[a_1+\delta,b_1]$.     Define $\tilde{f}_i$ as $f_i$ on every element of $\Lambda^i\setminus \{J\}$ and  on $[a_1+\delta,b_1]$, as   $\tilde{\pi}_\ell\circ \tilde{E}_\ell \circ \pi_\ell^{-1}$ on $\pi_\ell(R_\ell)$, for every $1\leq \ell\leq n$, and as   $\pi_{\ell+1}\circ  \omega_\ell \circ \tilde{\pi}_\ell^{-1}$ on $\tilde{\pi}_\ell(\tilde{R}_\ell)$ for every $0\leq \ell\leq n$. We have $|\tilde{f}_i'|\geq \alpha^{2j}$ everywhere.

Let  $Y_0=\tilde{R}_0, Y_1, \dots,Y_{2(j-n)}$ be pairwise disjoint intervals that are outside $\tilde{I}_i$ and such that
$$\Lambda_{i+1}= \tilde{\Lambda}_i\cup \{ Y_1, \dots,Y_{2(j-n)}\}$$
is a partition of an interval $I_{i+1}$ and $|Y_i|=  \alpha^{i}   |\tilde{R}_0|$ for $i\leq 2(j-n)$.  Define $f_{i+1}$ as $\tilde{f}_i$ in $\tilde{I}_i\setminus \tilde{R}_0$, as the orientation preserving  affine map on $Y_i$ satisfying   $f_{i+1}(Y_i)=Y_{i+1}$,  for each $i< 2(j-n)$ and  as $f_{i+1}(x)= \tilde{f}_i(\theta(x))$  on $Y_{2(j-n)}$, where $\theta$ is the orientation preserving affine map such that $\theta(Y_{2(j-n)})=Y_0=\tilde{R}_0$, in particular $|\theta'|=\alpha^{-2(j-n)}$ and consequently 
$|f_{i+1}'(x)|= |\tilde{f}_i'(\theta(x))|\alpha^{-2(j-n)} \geq  \alpha^{2n}\geq \alpha^2.$

We conclude that  $|f_{i+1}'|\geq  \alpha$ everywhere.  This completes the recursive construction of $f_k$. 

Due (\ref{small}) we have $|I_{I+1}\setminus I_i|\leq 2^{-i} |I|$, so $|I_k|\leq 2|I|$ and $I_k\subset \hat{I}$. To conclude the proof of the claim, take 
$$A=(\Lambda_j\setminus \{J_1,\dots,J_k\})\cup \{U_1,\dots,U_k\},$$
 $B= I_k\setminus A$, $F_1$ equal to $f^j$ on $A$ and $F_2$ equals to $f_k^{2j+1}$ on $B$.

\end{proof}

\section{Generic piecewise expanding maps on $\mathbb{R}^D$} \label{pie}
Piecewise smooth expanding maps on $\mathbb{R}^D$  received a lot of attention. See G\'ora  and Boyarsky \cite{MR1029902}, Adl-Zarabi \cite{ad} and Saussol \cite{saussol}. The goal of this section is to obtain generic results for piecewise expanding maps on $\mathbb{R}^D$ as in Cowieson \cite{cowieson} \cite{cowieson1}. 

Let  ${I}_r$ be a finite piecewise partition made of  open sets   in $\mathbb{R}^D$, and $m\geq 1$.  Consider the set $\mathcal{D}^m_{exp}(\{I_r\}_r)$ of maps $F\colon I \rightarrow I$ so that 
\begin{itemize}
\item[A.] for each $r$ the map  $F\colon I_r \rightarrow I$ extends as a   $C^m$ diffeomorphism on  a open neighborhood of $\overline{I}_r$,
\item[B.]  there is $\lambda > 1$ such that $|f'(x)\cdot w| \geq  \lambda |w|$ for every $x\in I$, $w\in \mathbb{R}^D$.
\end{itemize}

One can give a topology on $\mathcal{D}^m_{exp}(\{I_r\}_r)$ considering the product  topology of the  $C^r$ topologies on each branch $F\colon I_r\rightarrow I$.

The remarkable result by Cowieson \cite{cowieson} \cite{cowieson1} shows  that in  an open and dense set of maps   $F \in \mathcal{D}^m_{exp}(\{I_r\}_r)$, with $m\geq 2$ and the partition  $\{I_r\}_r$ is a $C^m$ partition (see  Cowieson \cite{cowieson} for details)  then  the map $F$ has an absolutely continuous invariant probability and whose density has bounded variation. We improve this result with
\begin{theorem}\label{gen}  Let  $\{I_r\}_r$ be  a finite partition of a good $C^1$ domain in $\mathbb{R}^D$ such that every $I_r$ is a $N$-good $C^1$ domain with a regular Whitney stratification.  For a map $F$ in an open and dense subset of $\mathcal{D}^{1+\beta+\epsilon}_{exp}(\{I_r\}_r)$, with $\beta, \epsilon   > 0$, $0 < s < \beta < 1/D$,  the Perron-Frobenious operator $\Phi\colon \mathcal{B}^s_{1,q}\rightarrow \mathcal{B}^s_{1,q}$ is quasi-compact with
\begin{equation} \label{eer} \sigma_{ess}(\Phi, \mathcal{B}^s_{1,q})\leq (\inf_x \min_{|v|=1} |D_xF\cdot v|)^{-Ds}.\end{equation} 
and  it satisfies the Lasota-Yorke inequality for the pair $(L^1,\mathcal{B}^s_{1,q})$. Moreover every absolutely continuous invariant probability of $F$  has a density that is  $\mathcal{B}^s_{1,q}$-positive, so in particular its support is an open set of $I$ (up to a subset of zero Lebesgue measure).
\end{theorem} 
\begin{proof} Let $p\in [1,\infty)$. By a transversality argument  as in Cowieson \cite{cowieson} \cite{cowieson1} one can prove that for an  open and dense set of maps $F$  in $\mathcal{D}^{1+\beta+\epsilon}_{exp}(\{I_r\}_r)$  the map $F^n$ is a piecewise expanding $C^{1+\beta+\epsilon}$ map defined in a partition $\{I_r^n\}_{r\in \Lambda_n}$  that consists of $B$-good $C^1$ domains with a regular Whitney stratification, for some $B$ that does not depend on $n$. Denote  by $h_r$ the corresponding inverse branches. By Proposition \ref{def} there is $\delta > 0$ such that if $Q\in \mathcal{D}$, $diam \ Q < \delta$ and $Q \subset I_r^n$ then $h_r^{-1}(Q)$ is  a
 $$(1-sp, \Crr{esssss}  (\Pi_{i\neq 1}  \frac{\alpha_i(Q)}{\alpha_1(Q)})^{sp}, \Crr{de000d}^{(1-\delta)(1-Dsp)})$$ regular domain in $(K,m,\mathcal{D})$. Here $0< \alpha_1(Q)\leq \alpha_2(Q) \leq \cdots \leq \alpha_n(Q)$ are such that $\{\alpha_i^2(Q)\}_i$ are  the eigenvalues of $A_QA^\star_Q$, with $A_Q=D_{x_Q}h_r^{-1}$ and $x_Q$ is an arbitrary element of $Q$. Here we must choose  $x\in Q$. Replacing $\{I_r^n\}_{r\in \Lambda_n}$ by an appropriated finer partition and increasing $B$ if necessary we can additionally  assume that $diam \ I_r^n < \delta$ for every $r\in \Lambda_n$ and  that for every $Q\in \mathcal{D}$, satisfying $Q \subset I_r^n$, $r \in \Lambda_n$ and $n$  we have that 
 $h_r^{-1}(Q)$ is  a
 $$(1-sp, \Cll{see}  (\Pi_{i\neq 1}  \frac{\alpha_i(r)}{\alpha_1(r)})^{sp}, \Crr{de000d}^{(1-\delta)(1-Dsp)})$$ regular domain in $(K,m,\mathcal{D})$. Here $0< \alpha_1(r)\leq \alpha_2(r) \leq \cdots \leq \alpha_n(r)$ are such that $\{\alpha_i^2(r)\}_i$ are  the eigenvalues of $A_rA^\star_r$, with $A_r=D_{x_r}h_r^{-1}$ and $x_r$ is an arbitrary element of $I_r$. So we can take
$$ \Crr{DGD1}(r)=\Crr{see}  (\Pi_{i\neq 1}  \frac{\alpha_i(r)}{\alpha_1(r)})^{sp}$$
 and
 $$ \Crr{DGD2}(r)= \Crr{de000d}^{(1-\delta)(1-Dsp)}.$$
 Notice that  (refining $\{I_r^n\}_r$ once again) 
 \begin{equation}\label{estiii} \frac{|Q|}{|h^{-1}_r(Q)|}   \leq \Cll{new} \Pi_{i}  \frac{1}{\alpha_i(r)} .\end{equation} 
 for some $\Crr{new}$ and every $Q\in I_r^n$, $r \in \Lambda_n$, with $Q\in \mathcal{D}$. On the other hand we have
 $$ \frac{1}{\Cll{new2}} \big(\frac{1}{2}\big)^{D(k_0(Q)-k_0(h_r^{-1}(Q))}  \leq \frac{1}{\alpha_1(r)^{D}} \leq \Crr{new2} \big(\frac{1}{2}\big)^{D(k_0(Q)-k_0(h_r^{-1}(Q)))}$$
  for some $\Crr{new2}$. In particular
  
 \begin{equation}\label{estii} |k_0(h_r^{-1}(Q))-k_0(Q)| \geq a_r= \frac{|\ln \alpha_1(r)|}{\ln 2} - \frac{\ln \Crr{new2}}{D\ln 2}.\end{equation} 
and if $\alpha_1 > 1$ then 
  $$ \frac{|Q|}{|h^{-1}_r(Q)|} \leq \Cll{new3} \big(\Pi_{i\neq 1}  \frac{\alpha_1(r)}{\alpha_i(r)} \big) \big(\frac{1}{2}\big)^{D|k_0(h_r^{-1}(Q))-k_0(Q)|}.$$
Take   $\Crr{DC2}= 1/2^{D}$ and 
 $$\Crr{DC1}=  \Crr{new3} \big(\Pi_{i\neq 1}  \frac{\alpha_1(r)}{\alpha_i(r)} \big) .$$

Since the Jacobian $g_r(x)= |Det \ D h_r|$ is  $(\beta+\epsilon)$-H\"older and $F$ is piecewise expanding, one can use the same argument as in the proof of Theorem \ref{marko} to conclude that 
$$|g_r1_W|_{\mathcal{B}^\beta_{p,q}(W,\mathcal{P}_W,\mathcal{A}^{sz}_{p,q})}  
\leq \Crr{DRP}(r) \Big(\frac{|Q|}{|h_r^{-1}(Q)|}\Big)^{1/p-s+\epsilon}   |W|^{1/p-\beta},$$
provided $W\in h_r^{-1}Q$, with $Q\subset I_r^n$, where
\begin{equation}\label{aqaq} \Crr{DRP}(r) =  \Crr{au} \Big(\Pi_{i}  \frac{1}{\alpha_i(r)}\Big)^{1-(1/p-s+\epsilon)},\end{equation}
and  $\Crr{au}$ does not depend on $r\in \Lambda_n$ and $n$. Here we may need to refine the partition  $\{I_r^n\}_{r\in \Lambda_n}$ again. Finally we obtain

\begin{align}\label{thetaest}
\Theta_r&= \Crr{DC1}^{\epsilon}(r) \Crr{DRP}(r)  \Crr{DGD1}^{1/p}  (1/2)^{D\epsilon a_r (1-\gamma_{_{DRS3}})}\nonumber  \\
 &\leq \Cll{new4} \big(\Pi_{i\neq 1}  \frac{\alpha_1(r)}{\alpha_i(r)} \big)^{\epsilon -s} \Big(\Pi_{i}  \frac{1}{\alpha_i(r)}\Big)^{1-(1/p-s+\epsilon)} \frac{1}{\alpha_1(r)^{D\epsilon (1-\gamma_{_{DRS3}})}} \nonumber \\
&\leq \Crr{new4}  \frac{1}{\alpha_1(r)^{D(s-\epsilon)} }  \Big(\Pi_{i}  \frac{1}{\alpha_i(r)}\Big)^{1-1/p} \frac{1}{\alpha_1(r)^{D\epsilon (1-\gamma_{_{DRS3}})}} \nonumber \\ 
&\leq \Crr{new4}  \Big(\Pi_{i}  \frac{1}{\alpha_i(r)}\Big)^{1-1/p} \frac{1}{\alpha_1(r)^{D(s- \epsilon\gamma_{_{DRS3}}) } }.
\end{align}

In particular for $p=1$ we have
$$ \Theta_r\leq  \Crr{new4} \alpha_1(r)^{-D(s- \epsilon\gamma_{_{DRS3}}) } \leq (\inf_x \min_{|v|=1} |D_xF\cdot v|)^{-Dn(s- \epsilon\gamma_{_{DRS3}})} .$$

Let  $t_n$ be such that 
$$M_n=\sup_{\substack{P\in \mathcal{P}^k \\ k\geq t_n}}\# \{r\in \Lambda_n \colon I_r^n\cap P\neq \emptyset    \} \leq B < \infty$$
for every $n$.  Due Corollary  \ref{Ngood} we can increase $t_n$ such that every $I_r^n$, with $r\in \Lambda_n$, is a 
$(1-\frac{1}{D},\Crr{srd2},t_n)$-strongly regular  domain.

We can apply Theorem \ref{trans-pa2} in  \cite{smania-transfer} with $p=1$ and 
$$T\leq B (\inf_x \min_{|v|=1} |D_xF\cdot v|)^{-Dn(s- \epsilon\gamma_{_{DRS3}})}$$ 
and Corollary \ref{trans-ww}  therein we  conclude that $\Phi^n$ has a $(\Cll[name]{DRSFR},\Cll[name]{DRSES})$-essential slicing  with 
$$\Crr{DRSES}\leq \Cll{new6} (\inf_x \min_{|v|=1} |D_xF\cdot v|)^{-Dn(s- \epsilon\gamma_{_{DRS3}})}$$
so it has a a similar upper bound for its essential spectral radius bounded and consequently (since $\epsilon$ can be taken arbitrarily small) 
$$\sigma_{ess}(\Phi, \mathcal{B}^s_{1,q})\leq (\inf_x \min_{|v|=1} |D_xF\cdot v|)^{-Ds}.$$

Moreover by Theorem \ref{trans-t1}  in \cite{smania-transfer} shows that the operator $\Phi^n$ satisfies the Lasota-Yorke inequality for the pair  $(L^1,\mathcal{B}^s_{1,q})$ and  Theorem \ref{trans-acim} proves that the support of the invariant measure is an open set up to a set of zero Lebesgue measure.
\end{proof}

\begin{remark}\label{winky} It is easy to construct examples of piecewise smooth metric-expanding maps for which that transfer operador is  quasi-compact.  For instance, consider the square $[0,1]^2$ with the dyadic good grid $\mathcal{D}$. Choose $k_0$, $\{I_r\}_r=\mathcal{D}^{k_0}$ and consider 
$$F\colon \cup_r I_r \rightarrow  [0,1]^2$$
such that $F\colon I_r \rightarrow Q_r$ is an affine bijection and $Q_r$ is choose to be one of the rectangles in the l.h.s. of Figure 2 (the largest square in the picture is $[0,1]^2$). If $k_0$ is large enough then the  "Winky Face" map  $F$ is piecewise metric-expanding map that satisfies the conclusions of Theorem \ref{gen}.\end{remark} 

Nakano and  Sakamoto \cite{ns} proved the quasi-compactness of the transfer operator of smooth expanding maps on manifolds with {\it no discontinuities} and they gave an estimate to the essential spectral radius. The estimate in  (\ref{eer}) is quite similar to their estimate in that case. One of main features of our methods it that it allows us to give a very good description of the support of the invariant measure. This is quite rare in high-dimensional settings, except if some additional  transitivity assumption holds. 

One may wonder  if the results for {\it all} piecewise expanding affine maps by Buzzi \cite{affineb} in the plane  and Tsujii \cite{tad} (in $\mathbb{R}^D$), as well   results for all  piecewise expanding real analytic maps by Buzzi \cite{MR1764923} and Tsujii \cite{tbv} can be generalised  for Besov spaces.

Consider solenoidal attractors as studied in Tsujii \cite{sole1} and Avila, Gou\"ezel  and Tsujii \cite{sole2}. This is an interesting case of study since   these  maps are  measure-expanding  but not metric-expanding maps.

Avila, Gou\"ezel  and Tsujii \cite{sole2} proved that  for generic solenoidal attractors  the support of the absolutely continuous  invariant measure of  a generic solenoidal attractor has non empty interior and  Bam\'{o}n,  Kiwi, Rivera-Letelier and Urz\'{u}a \cite{kiwi} proved that on certain  conditions the support  is an open set. It is an interesting question if one can use atomic decomposition methods to study transfer operators in this setting.

\vspace{1cm}
\centerline{ \bf V. COMPARING $\mathcal{B}^s_{p,q}$ WITH OTHER FUNCTION SPACES IN LITERATURE.}
\addcontentsline{toc}{chapter}{\bf V. COMPARING $\mathcal{B}^s_{p,q}$ WITH OTHER FUNCTION SPACES IN LITERATURE.}
\vspace{1cm}

Bellow we  show to the reader that the class of observables for which we obtain results for the quasi-compactness of the transfer operator and also good statistical properties is quite wide, and often include previous function spaces that appears in the literature on transfers operators. Note that we do not claim that functions of Besov spaces on measure spaces with good grid always have good statistical properties for {\it all} the dynamical systems the cited authors took under consideration.

\section{Keller's  spaces} The most influential result in  the study of transfer operators for potentials with low regularity ($p$-bounded variation  potentials) in one-dimensional dynamics was obtained by  Keller \cite{keller}. Following Keller's notation, given an interval $I$ and $y \in I$  define $S_\epsilon(y)=\{x\in I\colon \ |x-y|< \epsilon\}.$ and
$$osc(h,\epsilon,y)= ess \ sup \{ |h(x_1)-h(x_2)|, (x_1,x_2)\in S_\epsilon(y)\times S_\epsilon(y)\},$$
where the essential sup is taken with respect to the Lebesgue measure on $I\times I$. Define
$$OSC_1(h,\epsilon)= \int osc(h,\epsilon,y) \ dm(y),$$
and define 
$$var_{1,1/p}(h)=\sup_{0<\epsilon \leq  1} \frac{OSC_1(h,\epsilon)}{\epsilon^{1/p}}.$$
The Keller's space of functions of generalized bounded variation $BV_{1,1/p}$  is the space of functions in $L^1(m)$ with the norm
$$||f||_{BV_{1,1/p}}= ||f||_{L^1(m)} + var_{1,1/p}(f).$$
Keller considered  a piecewise expanding map of the interval $I$ with a finite partition  and such that $1/f'$ has  $p$-bounded variation. He proved the such transfer operator acts as a quasi-compact  operator on $BV_{p,1/p}$. The relation with our Banach spaces is given by

\begin{proposition} Let $\mathcal{D}$ be the dyadic  partition of the interval $I=[0,1]$ and $0<s <1$. We have $BV_{1,s} \subset \mathcal{B}^{s}_{1,\infty}(\mathcal{D})$. Moreover the corresponding  inclusion is continuous.
\end{proposition}
\begin{proof}
 Indeed, let $f\in BV_{1,s} $. If $x_Q\in Q$, where $Q$ is an interval, Then
\begin{align} osc_1(f,Q)&=\inf_c \int_Q |f(x)-c| \ dm(x)\leq \frac{1}{|Q|}\int_Q \int_Q |f(x)-f(y)| \ dm(x)dm(y)\nonumber \\
&\leq    \frac{1}{|Q|}\int_Q \int_Q osc(f,|Q|,x) \ dm(x)dm(y)=   \int_Q osc(f,|Q|,x) \ dm(x) \nonumber
\end{align}
 so we have 
\begin{align}
\sum_{Q\in \mathcal{D}^i} |Q|^{-s} osc_1(f,Q)    &\leq   2^{is} \sum_{Q\in \mathcal{D}^i}  \int_Q osc(f,|Q|,x) \ dm(x)  \nonumber \\
&\leq   2^{is}  OSC_1(f,2^{-i}) \nonumber \\   
&\leq  2^{s}  ||f||_{BV_{1,s}} \nonumber.
\end{align}
We complete the proof applying Theorem  \ref{besov-alte}.
\end{proof}

Note that $BV_{1,s} \not= \mathcal{B}^{s}_{1,\infty}(\mathcal{D})$ since we know (Keller \cite{keller})  that $BV_{1,s}\subset L^\infty(m)$ and there are unbounded  functions in the Besov space $\mathcal{B}^{s}_{1,\infty}(\mathcal{D})$.  Saussol \cite{saussol} used Keller's approach to study transfer operator in higher dimensions.

\section{Liverani's spaces} Liverani \cite{liverani} presented a new approach to deal with one-dimensional transfer operators with low regularity. His methods were extended successfully to some higher dimensional settings.  For every $s\in (0,1)$  he consider the Banach space $\mathcal{B}_s$ of all complex-valued borelian functions  in $L^{1/s}(I,m)$, with $I=[-\pi,\pi]$, for which the norm 
\begin{equation}\label{bs} |f|_{\mathcal{B}_s}= \sup \{ \big| \int g' f\ dm \big| \ \colon   g \in C^1, |g|_{C^s(I)}\leq 1  \}\end{equation}
is finite.  Liverani consider certain classes of piecewise Holder continous potentials and piecewise expanding maps on $I$ with an infinite partition and gave estimate to the essential spectral radius of  the associated transfer operator acting  on $\mathcal{B}_s$.  

\begin{proposition} Let $\mathcal{D}$ be the good grid of  dyadic  partitions of the interval $I=[0,1]$ and $0<s <1$. We have $\mathcal{B}^{s}_{1,1} \subset \mathcal{B}_{1-s}\subset \mathcal{B}^{s}_{1,\infty} $. Moreover these inclusions are continuous. 
\begin{proof} Let $f\in \mathcal{B}^{s}_{1,1}(\mathcal{D})$ and 
$$f= \sum_{k} \sum_{Q\in \mathcal{P}^k} c_Q  a_Q$$
be a  $\mathcal{B}^{s}_{1,1}(\mathcal{D})$-representation of $f$.  Let $g \in C^1, |g|_{C^{1-s}(I)}\leq 1 $. Then
\begin{align} |\int g' f \ dm |&=  |\sum_{k} \sum_{Q\in \mathcal{P}^k} c_Q \int g' a_Q \ dm| \nonumber \\
&\leq   \sum_{k} \sum_{Q\in \mathcal{P}^k} |c_Q| |Q|^{s-1}|\int_Q g'  \ dm| \nonumber \\
&\leq   \sum_{k} \sum_{Q\in \mathcal{P}^k} |c_Q|.
\end{align} 
So $|f|_{\mathcal{B}_{1-s}}\leq |f|_{\mathcal{B}^{s}_{1,1}}.$ 
Note that for the dyadic partition $\mathcal{D}$ and Lebesgue measure $m$,  the  Haar basis indexed by $\hat{\mathcal{H}}(I)= \{ I\}\cup_{Q\subset I} \mathcal{H}_Q$ constructed in Section \ref{besov-haar} in \cite{smania-besov}  is just the classical Haar basis, since  for every $Q\in \mathcal{D}^k$ there are exactly two intervals $Q_1,Q_2 \in \mathcal{D}^{k+1}$ such that $Q_i\subset Q$ satisfying   $\mathcal{H}_Q=\{ S^Q  \}$, where $S=( S_1^Q,S_2^Q)$, with $S_i^Q =\{Q_i\}$. Then
$$\phi_{S^Q} =  \frac{1}{|Q|^{1/2}}  \Big( 1_{Q_1}-1_{Q_2} \Big).$$
Let $f \in \mathcal{B}_{1-s}$. Then $f\in L^{1/{(1-s)}}(m)$, so 
$$f = \sum_{S\in \hat{\mathcal{H}}(I)} d_S \phi_S,$$
where
$$d_S= \int f \phi_S \ dm.$$
if $d_S\geq 0$ set  $\hat{\phi}_S=\phi_S$, otherwise let  $\hat{\phi}_S=-\phi_S$. Define
$$\psi_k = \sum_{Q\in \mathcal{D}^k} \hat{\phi}_{S^Q}.$$
Then
$$\int f \psi_k \ dm =\sum_{Q\in \mathcal{D}^k} |d_{S^Q}|.$$
Note that $\psi_k$  is the derivative of the Lipschitz function 
$$\Psi_k (x)=\int_0^x \psi_k \ dm,$$
and satisfies $|\psi_k|_\infty\leq 2^{-k/2}.$ Let $x, y \in I$, with $x\leq y$ and 
$$[a,b]=\overline{ \bigcup_{\substack{_{Q\in \mathcal{D}^k}\\ _{Q\subset [x,y] }  } }Q}.$$
Since 
$$\int_Q \psi_k \ dm = \int_Q \hat{\phi}_{S^Q} \ dm =0$$
and $|Q|=2^{-k}$  for every $Q\in \mathcal{D}^k$   we have that 
\begin{align} |\Psi_k (y)-\Psi_k (x)|&\leq |\Psi_k (y)-\Psi_k (b)|+|\Psi_k (a)-\Psi_k (x)\nonumber \\
&\leq  2^{k/2}(|y-b|+|z-x|)^{s}(|y-b|+|z-x|)^{1-s}\nonumber \\
&\leq  2^{k/2+s-ks}|y-x|^{1-s}, \nonumber  \end{align}
The last inequality follows from $\max \{ |y-b|,|z-x|\}\leq 2^{-k}$. 
So $$|\psi_k|_{C^{1-s}(I)}\leq~C~2^{(1/2-s)k}$$ and  by Liverani \cite{liverani}
$$\sum_{Q\in \mathcal{D}^k} |d_{S^Q}|= \int f \psi_k \ dm  \leq C 2^{(1/2-s)k}|f|_{\mathcal{B}_{1-s}}.$$
so 
$$ \sum_{Q\in \mathcal{D}^k} |d_{S^Q}||Q|^{1/2-s} \leq C   |f|_{\mathcal{B}_{1-s}}.$$
So by Theorem  \ref{besov-alte} in \cite{smania-besov} we have that $f \in \mathcal{B}_{1,\infty}^{s}$. 


\section{Thomine's   result for  Sobolev spaces}  Thomine \cite{thomine}  studied the action of the transfer operator of certain piecewise expanding maps on $\mathbb{R}^n$ on the classical Sobolev spaces 
$$\mathcal{H}^s_p(\mathbb{R}^n)=\{ u\in  L^p(\mathbb{R}^n)\colon      \mathcal{F}^{-1}((1+|\eta|^2)^{s/2} (\mathcal{F}u)(\eta))  \in L^p(\mathbb{R}^n)  \}.$$

Here $\mathcal{F}$ is  the Fourier transform on $\mathbb{R}^n$. It is well known that $\mathcal{H}^s_p(\mathbb{R})$ coincides with the Besov space $\mathcal{B}^s_{p,p}(\mathbb{R})$, whose restriction to the interval $I=[0,1]$   is exactly $\mathcal{B}^s_{p,p}(\mathcal{D})$, where $\mathcal{D}$ is the sequence of  dyadic partitions of $[0,1]$.

\section{Butterley's spaces}  Let $I$ be an  interval. Butterley \cite{butter} proved the quasi-compactness of  certain transfer operators with piecewise H\"older potentials acting on a Banach space $\mathfrak{B}_s \subset L^1(m)$, where $m$ is the Lebesgue measure on $I$, and

$$||f||_{\mathfrak{B}_s}:=\inf \{\sup_{0< k\leq 1} \Big(  k^{-s}|f_k-f|_{L^1(m)}+k^{1-s}||f_k||_{BV(I)} \Big)       \}$$
where the infimum runs over all possible families $\{f_k  \}_{0<k\leq 1}\subset L^1(m)$. 

\begin{proposition} We have $\mathfrak{B}_s\subset \mathcal{B}^s_{1,\infty}(\mathcal{D})$. Moreover the corresponding  inclusion is continuous.
\end{proposition}
\begin{proof}
 Indeed, let $f\in \mathbb{B}_s$. So choosing $k=2^{-i}$ we can find a family $\{g_i  \}_{i\in \mathbb{N}} \subset L^1(m)$ such that 
$$|g_i-f|_{L^1(m)}\leq 2^{-is+1} ||f||_{\mathfrak{B}_s} \ and \ ||g_i||_{BV(I)} \leq 2^{i(1-s)+1} ||f||_{\mathbb{B}_s}.$$
for every $i \in \mathbb{N}$. So
\begin{align}
\sum_{Q\in \mathcal{D}^i} |Q|^{-s} osc_1(f,Q)    &\leq  2||f||_{\mathfrak{B}_s}  + 2^{is}  \sum_{Q\in \mathcal{D}^i}  osc_1(g_i,Q) \nonumber \\
&\leq  2||f||_{\mathfrak{B}_s}  + 2^{i(s-1)}  ||g_i||_{BV(I)} \nonumber \\   
&\leq  4||f||_{\mathfrak{B}_s}\nonumber.
\end{align}
We complete the proof applying Theorem  \ref{besov-alte} in \cite{smania-besov}.
\end{proof}

\bibliographystyle{abbrv}
\bibliography{bibliografia}

\end{document}